\documentclass{amsart}
\usepackage{amssymb}
\usepackage{amsfonts}
\usepackage{amsmath}
\usepackage{graphicx}
\newtheorem{theorem}{Theorem}[section]
\newtheorem*{thmDW}{The Denjoy-Wolff Theorem}
\newtheorem{lemma}[theorem]{Lemma}
\newtheorem{obser}[theorem]{Observation}
\newtheorem{prop}[theorem]{Proposition}
\newtheorem{cor}[theorem]{Corollary}

\newtheorem{ndef}[theorem]{Definition}
\newtheorem*{defin}{Definition}

\newtheorem*{PTC}{Parabolic-Type Test}
\theoremstyle{remark}
\newtheorem{example}[theorem]{Example}

\newtheorem*{ack}{Acknowledgment}

\makeatletter
\newcommand{\mysubsection}%
{\@startsection{subsection}{2}{\z@}{-3.25ex plus -1ex minus -.2ex}{-1ex}{\normalsize\sc}}
\makeatother

\newcommand{\Al}{\mathcal{A}}
\renewcommand{\phi}{\varphi}
\renewcommand{\epsilon}{\varepsilon}

\newcommand{\C}{\mathbb{C}}

\newcommand{\D}{\mathbb{D}}
\newcommand{\dw}{\omega}

\newcommand{\htwo}{H^2}

\renewcommand{\notin}{\not\in}

\renewcommand{\Re}{{\rm Re\,}}

\newcommand{\rhp}{R}

\newcommand{\vstrut}{\rule{0in}{0.18in}}

\newcommand{\ess}{\mathcal{S}}
\renewcommand{\mod}{{\rm mod\, }}
\begin{document}



\title[Spectra of composition operators]{Spectra of Composition Operators with Symbols in $\ess(2)$}
\author{Paul S. Bourdon}
\address{Department of Mathematics\\  University of Virginia\\ Charlottesville, VA 22903}
\email{psb7p@virginia.edu}

\begin{abstract}
{\scriptsize Let $H^2(\D)$ denote the classical Hardy space of the open unit disk $\D$ in the complex plane. We obtain descriptions of both the spectrum and essential spectrum of composition operators on $H^2(\D)$ whose symbols belong  to the class $\ess(2)$ introduced by Kriete and Moorhouse [{\it Trans.\ Amer.\ Math. Soc.}  359, 2007].  Our work  reveals new possibilities for the shapes of composition-operator spectra, settling a conjecture of Cowen's  [{\it J.\ Operator Th.} 9, 1983].   Our results depend on  a number of lemmas, perhaps of independent interest,  that provide spectral characterizations of sums of elements of a unital algebra over a field when certain  pairwise products of the summands are zero.}
\end{abstract}
\maketitle
 
 \section{Introduction}

 Let $\D$ be the open unit disc in the complex plane, let $H(\D)$  be the space of analytic functions on $\D$, and let $H^2(\D)$ be the classical Hardy space, consisting of those functions in $H(\D)$ whose Maclaurin coefficients are square summable.  For $\phi$ an analytic selfmap of $\D$, let $C_\phi$ be the composition operator with symbol $\phi$, so that $C_\phi f = f\circ \phi$ for any $f \in H(\D)$.   Clearly $C_\phi$ preserves $H(\D)$.  Littlewood \cite{Lit} proved $C_\phi$ also preserves $H^2(\D)$; and thus, by the closed-graph theorem, $C_\phi:H^2(\D)\rightarrow H^2(\D)$ is a bounded linear operator.   For the remainder of this paper, we assume all composition operators act on $H^2(\D)$.  
 
 Beginning in the late 1960s, through the 1970s, and early 1980s, Nordgren \cite{ND}, Deddens \cite{DD}, Caughran and Schwartz \cite{CSw}, Kamowitz \cite{KM}, and Cowen \cite{Cow2} characterized the spectrum of composition operators on $H^2(\D)$ whose symbols $\phi$ are linear-fractional selfmaps of $\D$.  These spectra take a variety of forms, e.g., disks, annuli, and spirals, depending on the location of the Denjoy-Wolff point $\omega$ of the symbol $\phi$, the derivative $\phi'(\omega)$,  whether or not $\phi$ is inner, and whether or not $\phi$ induces a composition operator is that is power-compact.  See \cite{Cow} for details.  In this paper, we use these known spectral characterizations for linear-fractional composition operators to obtain spectral characterizations for composition operators  whose symbols belong to the class $\ess(2)$ introduced by Kriete and Moorhouse \cite{KMH}.  When $\phi\in \ess(2)$,   $C_\phi$ is equivalent, in the Calkin Algebra, to a sum of linear-fractional composition operators, and this equivalence permits us to characterize the essential spectrum of $C_\phi$.   Obtaining the full spectrum, once the essential spectrum is known, is not difficult  for the composition operators we consider.  This paper may be viewed as a sequel to \cite{BSp}, which contains characterizations of spectra of certain composition operators $C_\phi$ under the assumption that $C_\phi$ is equivalent, in the Calkin Algebra, to a single linear-fractional composition operator.
 
    In the next section, we show that
\begin{equation}\label{lpmap}
\phi_{lp}(z) = \frac{2z^2 - z -2}{2z^2-3}
\end{equation}
belongs to the class $\ess(2)$ and that $C_{\phi_{lp}}-C_{\psi_1} - C_{\psi_2}$ is a compact operator on $H^2(\D)$, where
\begin{equation}\label{psi12}
\psi_1(z) = \frac{4-3z}{5-4z} \quad \text{and} \quad \psi_2(z) = \frac{41z+32}{40z + 49}.
\end{equation}
By \cite[Cor.\ 6.2]{Cow2}, the essential spectrum of $C_{\psi_1}$ is the segment $[0,1]$,  while the essential spectrum of $C_{\psi_2}$ is the disk $\{z: |z| \le 1/3\}$  \cite[Proof of Theorem 5] {Cow} (see also \cite[Theorem 3.2]{BSp}).  By Theorem~\ref{MT}  below, the essential spectrum and spectrum of $C_{\phi_{lp}}$ both equal the union of $[0,1]$ and $\{z: |z| \le 1/3\}$.  Thus the spectrum of $C_{\phi_{lp}}$ is shaped like a lollipop.   The selfmap $\phi_{lp}$ of $\D$ is of ``parabolic non-automorphism type'' (see Definition~\ref{Type}  below).  For $\phi$ of this type, Cowen \cite[Conjecture 4 of  \S 6]{Cow2} conjectures that the spectrum of $C_\phi$ is a region between two spirals; that is, for some $\theta_1, \theta_2$ with $-\pi/2 \le \theta_1 \le \theta_2 \le \pi/2$, the spectrum is  $\{e^{-\beta}: \theta_1 \le \arg \beta \le \theta_2\} \cup \{0\}$. A lollipop is not such a region and thus our work settles Cowen's conjecture in the negative.  

  In the next section, we present background information  needed for our work.  Section 3 contains some lemmas characterizing the spectrum of a sum  $a_1 + a_2 + \cdots + a_n$ of elements of a unital algebra  $\Al$ over a field, where certain  pairwise products with factors from the set  $\{a_1, a_2, \ldots, a_n\}$  are zero.  We rely on these lemmas in Section 4, which contains our main results.
  
  \begin{ack} The author is grateful to Thomas Kriete for many helpful conversations regarding $\ess(2)$ as well as the spectral lemmas of Section 3.  These lemmas are much improved over those appearing in an earlier version of this paper thanks to Trieu Le,  who pointed out that the original version of what is now Proposition~\ref{TA} had unnecessary hypotheses.   
  \end{ack} 
   \section{Background}
 
       We begin by describing function-theoretic properties of $\phi$ known to influence the spectral behavior of the composition operator $C_\phi: H^2(\D) \rightarrow H^2(\D)$.
   
   \subsection{Function-theoretic preliminaries}

   Throughout this paper $\phi$ denotes an analytic selfmap of $\D$ and  $\phi^{[n]}$ represents the $n$-th iterate of $\phi$, $\phi\circ \phi \circ  \cdots \circ \phi$, $n$ times ($\phi^{[0]}$ is the identity).   Let $n$ be a positive integer, let $\zeta\in \partial \D$, and let $0 \le \epsilon < 1$. Following
\cite[p.~50]{BoS}, we say that $\phi$ belongs to $C^{n+\epsilon}(\zeta)$ provided that $\phi$ is
differentiable at $\zeta$ up to order $n$ $($viewed as a function
with domain $\D\cup \{\zeta\})$ and, for $z\in \D$, has the expansion
$$
\phi(z) = \sum_{k=0}^{n}\frac{\phi^{(k)}(\zeta)}{k!}(z-\zeta)^k
+ \gamma(z),
$$ 
where $\gamma(z) = o(|z-\zeta|^{n+\epsilon})$ as $z\rightarrow
\zeta$ from within $\D$.  It is not difficult to show that $\phi\in C^{n}(\zeta)$ whenever $\phi^{(n)}$ extends
continuously to $\D\cup\{\zeta\}$.   

     The spectral properties of a composition operator are strongly tied to the location of the Denjoy-Wolff point of its symbol.   Recall that an elliptic automorphism is an automorphism of $\D$ fixing a point in $\D$. 
   
\begin{thmDW}  If $\phi$ is an analytic selfmap of $\D$ that is not an elliptic automorphism, then there is a point $\omega$ in the closed disk $\D^-$ such that 
$$
\text{for all}\ z\in \D, \ \ \phi^{[n]}(z) \rightarrow \omega \ \text{as}\ n\rightarrow \infty.
$$
The Denjoy-Wolff point $\omega$ of $\phi$ may be characterized as follows:
\begin{itemize}
\item[(i)] if $|\omega| < 1$, then $\phi(\omega) = \omega$ and $|\phi'(\omega)| < 1$;
\item[(ii)]  if $|\omega| = 1$, then $\phi(\omega) = \omega$ and $0 < \phi'(\omega) \le 1$.
\end{itemize}
\end{thmDW}

When (ii) holds,  $\phi(\omega)$ is the angular (nontangential) limit of  $\phi$ at $\omega\in \partial \D$ and $\phi'(\omega)$ represents the angular derivative, which may be computed as the angular limit of $\phi'$ at $\omega$.  Observe that if $\omega\in \partial \D$ is the Denjoy-Wolff point of $\phi$, then (ii) above yields, in particular, that $\phi'(\omega) > 0$.  This is a general property of angular derivatives at fixed points of $\phi$  that lie on $\partial \D$; that is, if $\zeta\in \partial \D$ is fixed by $\phi$ and the angular derivative of $\phi$ exists at $\zeta$, then it follows (from the Julia-Carath\'{e}odory theorem) that  $\phi'(\zeta)> 0$. For details about angular derivatives and the Julia-Carth\'{e}odory theorem, the reader may consult \cite[Chapter 4]{Sh2} or \cite[\S 2.3]{CMB}.  As an example, note $\phi_{lp}$ defined by (\ref{lpmap}) has Denjoy-Wolff point $1$: $\phi(1) = 1$ and $\phi'(1) = 1$. 

\begin{ndef}\label{Type}    Suppose that $\phi$ has Denjoy-Wolff point $\omega$.  We classify $\phi$ as follows $($cf. \cite[Definition 0.3]{BoS}$)$:
  \begin{itemize}
  \item   if $\omega \in \D$, we say $\phi$ is of {\it dilation type};
  \item if $\omega\in \partial\D$ and $\phi'(\omega) < 1$, we say $\phi$ is of {\it hyperbolic type};
  \item If $\omega\in \partial\D$ and $\phi'(\omega) = 1$, then $\phi$ is of {\it parabolic type}.  Furthermore, if the orbit $(\phi^{[n]}(0))$ has consecutive terms separated in the hyperbolic metric on $\D$, then $\phi$ is of {\it parabolic automorphism type}; otherwise, $\phi$ is of  {\it parabolic non-automorphism type.}
  \end{itemize}
\end{ndef}
Distinguishing the two subcases of parabolic type can be difficult, but for maps having sufficient smoothness on $\D \cup \{\omega\}$, there is the following test (\cite[Theorem 4.4]{BoS}):
\begin{PTC}  Suppose that $\phi\in C^2(\dw)$  and $\phi'(\dw) = 1$.  Then $\Re(\dw \phi''(\dw)) \ge 0$; moreover,
\begin{itemize}
\item[(a)]  if $\phi''(\dw) = 0$ or if $\Re(\dw \phi''(\dw)) > 0$, then $\phi$ is of parabolic non-automorphism type;
\item[(b)] if $\dw\phi''(\dw)$ is pure imaginary $($and nonzero$)$ and $\phi\in C^{3+\epsilon}(\dw)$ for some positive $\epsilon$, then $\phi$ is of parabolic automorphism type.
\end{itemize}
\end{PTC}
Note the Parabolic-Type Test says in particular that if $1$ is the Denjoy-Wolff point of a parabolic type $\phi$ and  $\phi\in C^2(1)$, then $\Re(\phi''(1)) > 0$ is sufficient to ensure that $\phi$ is of parabolic non-automorphism type.  

  To explore further the notion of type for analytic selfmaps $\phi$ of $\D$,  we consider the ``right-halfplane incarnation''   of such maps under conjugation by $R(z) =  (1+z)/(1-z)$, where $R$ maps $\D$ onto the right halfplane $\{z: \Re(z)>0\}$ with the boundary points $1$ and $-1$ taken, respectively, to $\infty$ and $0$.    Observe that $\phi$ is an analytic selfmap of $\D$ if and only if  $ \Phi:=R\circ\phi \circ R^{-1}$ is an analytic selfmap of the right halfplane.   For instance, the selfmap of $\D$
  $$
  \phi_{lp}(z) =  \frac{2z^2 - z -2}{2z^2-3}
  $$
 has right-halfplane incarnation $\Phi_{lp}:=R\circ \phi_{lp}\circ R^{-1}$ given by
\begin{equation}\label{rhplp}
  \Phi_{lp}(w) = w + 8 - \frac{8}{w+1}.
\end{equation}
  Observe that if $\Re(w) >0$, then clearly $\Re(\Phi_{lp}(w)) >0$; so this right-halfplane incarnation provides a simple way to see that $\phi_{lp}$ is indeed a selfmap of $\D$.  In general, if $\phi$ is selfmap of $\D$ having Denjoy-Wolff point $1$ and $\phi\in C^2(1)$, then the right-halfplane incarnation of $\Phi$ of $\phi$ will have the form
\begin{equation}\label{RHPI}
  \Phi(w) = \frac{1}{\phi'(1)} w +  \frac{1}{\phi'(1)} -1 + \frac{\phi''(1)}{\phi'(1)^2} + \Gamma(w),
  \end{equation}
  where $\Gamma(w) \rightarrow 0$ as $w\rightarrow \infty$ in the right halfplane (cf. \cite[Equation (27)]{BLNS}).   Thus for a selfmap $\phi$ of $\D$ that belongs to $C^2(1)$ and has Denjoy-Wolff point $1$, the coefficient of $w$ in its right-halfplane incarnation $\Phi$ reveals whether the map is of hyperbolic or parabolic type and, in the parabolic case, the constant term in (\ref{RHPI}), which reduces to $\phi''(1)$,  can indicate whether the map is of non-automorphism type.  For example,  using (\ref{RHPI}) and (\ref{rhplp}), we see that $\phi_{lp}''(1) = 8$ and the Parabolic-Type Test assures us that $\phi_{lp}$ is of non-automorphism type because $\Re(\phi_{lp}''(1)) > 0$.  The right-halfplane incarnation of $\phi$ also provides a convenient way to understand the terminology ``automorphism-type'' and ``non-automorphism type''  as well as the naturalness of the Parabolic-Type Test.  A parabolic automorphism of the disk with Denjoy-Wolff point  $\omega = 1$ is precisely a linear-fractional selfmap of $\D$ whose right-halfplane incarnation takes the form
  $$
  \Phi(w) = w + a,
  $$
  where $a$ is pure imaginary (nonzero).  If, on there other hand, $\Phi(w) = w + a$ where $a$ has positive real part, then  the corresponding $\phi$ would be a parabolic (linear-fractional) non-automorphism of $\D$ and $ \Re(\phi''(1))  =\Re(a) > 0$, consistent with the Parabolic-Type Test.  
  
The notion of the ``order of contact'' (or ``order of approach'')  that  $\phi(\D)$ has with $\partial \D$ has long played an important role in composition-operator theory (see, e.g. \cite[Section 3.6]{Sh2} for a discussion of order of contact and the compactness question). 

\begin{defin}  Following Kriete and Moorhouse \cite{KMH}, we say that $\phi$ has order of contact $2$ at $\zeta\in \partial \D$ provided that  $|\phi(\zeta)| =1$ and 
\begin{equation}\label{KMOOC}
\frac{1-|\phi(e^{i\theta})|^2}{|\phi(\zeta) - \phi(e^{i\theta})|^2}  
\end{equation}
is essentially bounded above and away from $0$ as $e^{i\theta} \rightarrow \zeta$.  
\end{defin}

 Let $C_r$ be a circle of radius $r < 1$ that is internally tangent to $\partial D$ at $\phi(\zeta)$.  The point $\phi(e^{i\theta})$ in $\D$ lies on $C_r$ for some $\theta$ if and only if the quantity (\ref{KMOOC}) equals $(1-r)/r$. Thus, geometrically speaking, $\phi$ has order of contact $2$ at $\zeta$ means that there are two circles, each internally tangent to $\partial \D$ at $\phi(\zeta)$ and a subset $E$ of  $\partial \D$ having full measure such that if  $e^{i\theta}\in E$ is sufficiently close to $\zeta$, then $\phi(e^{i\theta})$ lies between the two circles.  It follows that if $\phi$ is a linear-fractional non-automorphism such that $\phi(\zeta) \in \partial \D$, then $\phi$ has order of contact $2$ at $\zeta$.  
 
 For  selfmappings $\phi$ of $\D$ that belong to $C^2(\zeta)$,  we establish the following easy test (cf. \cite[p.\ 49]{BLNS}, \cite[Propositions\ 1.2, 1.3]{BSp}) involving $\phi'(\zeta)$ and $\phi''(\zeta)$ that reveals whether $\phi$ has order of contact 2 at $\zeta$.

\begin{prop}\label{OCT}  Suppose that  $\phi\in C^2(\zeta)$ and that $\phi(\zeta) = \eta \in \partial \D$.   Then $\phi$ has order of contact $2$ at $\zeta$ if and only if
\begin{equation}\label{CurvatureCondition}
\Re\left(\frac{1}{|\phi'(\zeta)|} + \frac{\zeta \phi''(\zeta)}{\phi'(\zeta)|\phi'(\zeta)|} - 1 \right) > 0.
\end{equation}
\end{prop}

\begin{proof}  Set   $\psi(z)= \bar{\eta} \phi(\zeta z)$ and note $\psi(1) = 1$ because $\phi(\zeta) = \eta$.  Observe that 
$$
\frac{1-|\phi(e^{i\theta})|^2}{|\phi(\zeta) - \phi(e^{i\theta})|^2} = \frac{1-|\psi(\bar{\zeta}e^{i\theta})|^2}{|1 - \psi(\bar{\zeta}e^{i\theta})|^2}.
$$
Thus, $\phi$ satisfies the order of contact 2 condition at $\zeta$ if and only if $\psi$ satisfies the order of 2 contact condition at its fixed point $1$.  Also note that $\psi$ is $C^2(1)$ if and only if $\phi\in C^2(\zeta)$.  We show that $\psi$ has order of contact $2$ at $1$ if and only if 
$
\Re\left(\frac{1}{\psi'(1)} + \frac{\psi''(1)}{\psi'(1)^2} - 1\right) > 0.
$
This will complete the proof because  $\psi'(1) = \zeta\bar{\eta}\phi'(\zeta)$ is positive (a consequence of the Julia Carath\'{e}odory Theorem; see, e.g., \cite[Theorem 2.44]{CMB}), making $\psi'(1)= |\phi'(\zeta)|$; moreover,
$$
 \frac{\psi''(1)}{\psi'(1)^2}  = \frac{\zeta^2 \bar{\eta} \phi''(\zeta)}{\zeta^2 \bar{\eta}^2\phi'(\zeta)^2} = \frac{\zeta \phi''(\zeta)}{\phi'(\zeta)|\phi'(\zeta)|}.
 $$
Let $\Psi= R \circ \psi \circ R^{-1}$ be the right-halfplane incarnation of $\psi$. Let $z\in \D$ and  $w = R(z)$; we have
$$
\psi(z)  = R^{-1}(\Psi(w)) = \frac{\Psi(w) - 1}{\Psi(w)+1}.
$$
Hence 
\begin{align*}
\frac{1-|\psi(z)|^2}{|1 - \psi(z)|^2} &= \frac{|\Psi(w) + 1|^2 - |\Psi(w) -1|^2}{4} \\
& = \Re(\Psi(w)). \nonumber
\end{align*}
Because  $\psi\in C^2(1)$, we may replace $\Psi$ with its expansion 
$$
\Psi(w) = \frac{1}{\psi'(1)}w + \frac{1}{\psi'(1)} + \frac{\psi''(1)}{\psi'(1)^2} - 1 + \Gamma(w),
$$
where $\Gamma(w) \rightarrow 0$ as $w \rightarrow \infty$ (where $w$ is in the right-halfplane).  We obtain
\begin{equation}\label{GTOOC}
\frac{1-|\psi(z)|^2}{|1 - \psi(z))|^2}   = \Re\left( \frac{1}{\psi'(1)}w + \frac{1}{\psi'(1)} + \frac{\psi''(1)}{\psi'(1)^2} - 1 + \Gamma(w)\right),
\end{equation}
where $w = R(z)$.

For all $\xi$ belonging to some subset $E$ of $\partial \D$ having full measure,  $\psi$ has nontangential limit  $\ne 1$ at $\xi$ while necessarily then $\Psi$ has nontangential limit $R(\psi(\xi))$ at $\rhp(\xi)$ (approach from the right halfplane).  Because $\rhp(\xi)$ lies on the imaginary axis and $\psi'(1)$ is positive, we conclude from (\ref{GTOOC}) that for all $\xi\in E$ , we have
\begin{equation}\label{GTOOC2}
\frac{1-|\psi(\xi)|^2}{|\psi(1) - \psi(\xi))|^2} = \Re\left(\frac{1}{\psi'(1)} + \frac{\psi''(1)}{\psi'(1)^2} - 1 + \Gamma(R(\xi))\right).
\end{equation}
Because $\Gamma(w) \rightarrow 0$ as $w\rightarrow \infty$,  we must have $\Gamma(R(\xi)) \rightarrow 0$ as $\xi \in E$ approaches $1$.  Thus, from (\ref{GTOOC2}), we see
$$
\frac{1-|\psi(\xi)|^2}{|\psi(1) - \psi(\xi))|^2}  \rightarrow  \Re\left(\frac{1}{\psi'(1)} + \frac{\psi''(1)}{\psi'(1)^2} - 1\right) , \ \text{as}\ \xi\in E \ \text{approaches} \ 1.
$$
Thus  $\frac{1-|\psi(\xi)|^2}{|\psi(1) - \psi(\xi))|^2}$ is essentially bounded above and away from $0$, as $\xi$ in $E$ approaches $1$,  if and only if  $ \Re\left(\frac{1}{\psi'(1)} + \frac{\psi''(1)}{\psi'(1)^2} - 1\right) > 0$, as desired.
\end{proof}

Remarks:
\begin{itemize}
\item[(a)] Observe that if $\phi$ has Denjoy-Wolff point $\dw\in \partial \D$,  $\phi'(\dw) =1$ (parabolic case), and $\phi\in C^2(\dw)$, then the second-order contact condition (\ref{CurvatureCondition}) reduces to simply:
$$
\Re(\dw \phi''(\dw)) > 0 \quad \text{iff}\ \phi \ \text{has order of contact 2 at $\dw$}.
$$
Thus, by the Parabolic-Type Test, for  $\phi\in C^2(\dw)$ of parabolic type,  we see that $\phi$ having order of contact $2$ at its  Denjoy-Wolff point $\dw$ is equivalent to $\phi$ being of parabolic non-automorphism type.  
\item[(b)]The condition (\ref{CurvatureCondition}) ensuring second-order contact has an appealing interpretation in terms of curvature.  If $\phi$ is holomorphic in a neighborhood of $\zeta$, then 
$$
\Re\left(\frac{1}{|\phi'(\zeta)|} + \frac{\zeta \phi''(\zeta)}{\phi'(\zeta)|\phi'(\zeta)|}\right)
$$
is the curvature of the parametric curve $t\rightarrow \phi( e^{it})$ at $t=\arg(\zeta)$ (see, e.g., \cite[Chapter 5, Section IX]{NB}).  Thus the second order contact condition at $\zeta$ is saying that the curvature of the image of the unit circle under $\phi$ at $\zeta$ exceeds 1,  the curvature of the unit circle. 
\end{itemize}

\subsection{Aleksandrov-Clark measures and the class $\ess(2)$}\label{ClarkMeasures}

Kriete and Moorhouse \cite[Section  5]{KMH} introduce the ``sufficient-data''  class $\ess(2)$ consisting of analytic selfmaps $\phi$ of $\D$ having, roughly speaking,  limited contact with $\partial \D$ that is both order-2 and $C^2$.   The rigorous definition of $\ess(2)$ requires the notion of Aleksandrov-Clark measures.    

 For each $\alpha \in \partial \D$, 
$$
z \mapsto \Re\left(\frac{\alpha + \phi(z)}{\alpha - \phi(z)}\right) = \frac{1- |\phi(z)|^2}{|\alpha - \phi(z)|^2}, z\in \D,
$$
is a positive, harmonic function; thus, there is a finite positive Borel measure $\mu_\alpha$ on $\partial \D$, called the Aleksandrov-Clark measure of $\phi$ at $\alpha$,  such that for each $z\in \D$,
$$
 \Re\left(\frac{\alpha + \phi(z)}{\alpha - \phi(z)}\right) = \int_{\partial \D} P_z(\xi)\, d\mu_\alpha(\xi),
$$
where $P_z(\xi) = (1-|z|^2)/|\xi - z|^2$ is the Poisson kernel at $z$. A good reference for Aleksandrov-Clark measures is \cite{CMR}.    

Connections between the compactness of the composition operator $C_\phi$ on $H^2(\D)$ and properties of $\phi$'s Aleksandrov-Clark measures were developed by Sarason \cite{SA}, Shapiro and Sundberg \cite{ShSb}, and Cima and Matheson \cite{CMa}, with Cima and Matheson establishing that the essential norm $\|C_\phi\|_e$ of $C_\phi$ is given by
\begin{equation}\label{CMC}
\|C_\phi\|_e^2 = \sup\{\mu^s_\alpha(\partial \D): \alpha \in \partial \D\},
\end{equation}
where for each $\alpha \in \partial\D$,   $\mu_\alpha^s$ is the singular part of $\mu_\alpha$ in its Lebesgue decomposition.

 Let $\alpha\in \partial\D$. The support of $\mu^s_\alpha$ is contained in the closure of the set $\phi^{-1}(\{\alpha\})$, consisting of those points on the unit circle where $\phi$ has nontangential limit $\alpha$.    The measure $\mu^s_\alpha$ decomposes into a sum of a pure point measure 
\begin{equation}\label{ppd}
\mu_{\alpha}^{pp} = \sum_{\phi(\zeta)  =\alpha} \frac{1}{|\phi'(\zeta)|} \delta_\zeta
\end{equation}
and a continuous singular measure.  In the preceding formula for $\mu_\alpha^{pp}$, $\phi'(\zeta)$ is the angular derivative of $\phi$ at $\zeta$ taken to be $\infty$ if it does not exist.   Let 
$$
F(\phi) = \{\zeta\in \partial \D: \phi\ \text{has finite angular derivative at}\ \zeta\}
$$
and observe that by (\ref{ppd}), if $\zeta\in F(\phi)$, then  $\mu_{\phi(\zeta)}^{pp}$ is not the zero measure.  

Let 
$$
E(\phi) = \left( \cup_{\alpha\in \partial \D}\ \text{support}(\mu_\alpha^s)\vstrut \right)^-,
$$
so that $E(\phi)$ is the closure of the union of the (closed)  support sets of the singular parts of the Aleksandrov-Clark measures for $\phi$.    We will be concerned only with  $\phi$ such that  $E(\phi)$ is finite.  This means, in particular, that for each $\alpha \in \partial \D$, $\mu_\alpha^s$ is a pure point measure supported on  the set of those $\zeta\in \partial\D$ such $\phi$ has finite angular derivative at $\zeta$ and $\phi(\zeta) = \alpha$.  Thus, our assumption that $E(\phi)$ is finite means $F(\phi)$   is finite and
$$
E(\phi) = F(\phi).
$$

If $E(\phi)$ is empty, then $C_\phi$ is compact  by (\ref{CMC}).  Note that $E(\phi)$ is certainly empty if $\|\phi\|_\infty:= \sup \{|\phi(z)|: z\in \D\} < 1$; thus, $C_\phi$ is compact in this situation (for an elementary proof, see \cite[p.\ 23]{Sh2}).    

  Kriete and Moorhouse's sufficient-data class $\ess(2)$ consists of those analytic selfmaps $\phi$ of $\D$ such that 
  \begin{itemize}
  \item[(i)]  $\phi$ has radial limit of modulus less than $1$ at almost every point of $\partial \D$;
  \item[(ii)] $E(\phi)$ is finite (so that $E(\phi) = F(\phi)$);
  \item[(iii)] for each $\zeta\in E(\phi)$, $\phi$ has order of contact $2$ at $\zeta$;
  \item[(iv)] for each $\zeta\in E(\phi)$, $\phi\in C^2(\zeta)$ (so that $\phi$ has derivative data at  $\zeta$ sufficient to match its order of contact at $\zeta$). 
  \end{itemize}
We remark that it's possible to show condition (ii) above implies condition (i); for instance, an argument can based on the Aleksandrov Disintegration Theorem (see, e.g., \cite[Section 9.3]{CMR}).

\begin{defin}  Let $\phi\in C^2(\zeta)$ for some $\zeta\in \partial \D$.  The second-order data of $\phi$ at $\zeta$, denoted $D_2(\phi,\zeta)$, is given by
$$
D_2(\phi, \zeta) = (\phi(\zeta), \phi'(\zeta), \phi''(\zeta)).
$$
\end{defin}

  Generalizing results in \cite[Section 7]{BLNS}, Kriete and Moorhouse \cite[Corollary 5.16]{KM}\ establish  the following:
  
\begin{theorem}[Kriete-Moorhouse]\label{KMS2T}  Let $\phi\in \ess(2)$ with $E(\phi)= \{\zeta_1, \zeta_2, \ldots, \zeta_n\}$.  For $j = 1, \ldots, n$, let $\psi_j$ be the unique linear-fractional selfmap of $\D$ such that $D_2(\psi_j, \zeta_j) = D_2(\phi, \zeta_j)$.  Then 
\begin{equation}\label{KMR}
C_\phi = C_{\psi_1} +  \cdots + C_{\psi_n} + K,
\end{equation}
where $K$ is a compact operator on $H^2(\D)$.  
\end{theorem}

 Note well that the linear-fractional maps $\psi_j$ of (\ref{KMR}) are necessarily non-automorphic selfmaps of $\D$. One way to see this is to observe that $z\mapsto \overline{\phi(\zeta_j)}\psi_j(\zeta_j z)$ and $z\mapsto  \overline{\phi(\zeta_j)}\phi(\zeta_j z)$ have the same second-order data at their common fixed point $1$. Thus, the first two terms in their right-halfplane-incarnation expansions, $a w + b +\cdots$, agree;  moreover, $\Re(b) >0$ because $\phi$ has second-order contact at $\zeta_j$.  However, $aw + b$ is the complete right-halfplane incarnation of the linear-fractional map $z\mapsto \overline{\phi(\zeta_j)}\psi_j(\zeta_j z)$, and since $\Re(b) >0$, we see that $\psi_j$ is not an automorphism of $\D$.  
\begin{example} \label{MLPS2} The mapping $\phi_{lp}$ of (\ref{lpmap}) belongs to $\ess(2)$, with $E(\phi_{lp}) = F(\phi_{lp})= \{1, -1\}$.  Because $\phi_{lp}$ is analytic on the closed unit disk, we certainly have $\phi_{lp}\in C^2(1)\cap C^2(-1)$.  We have already noted that $\phi_{lp}$ is of parabolic non-automorphism type with Denjoy-Wolff point $1$; thus, $\phi$ has order of contact $2$ at $1$ by the remarks following the proof of Proposition~\ref{OCT}.  We employ Proposition~\ref{OCT} to show that $\phi_{lp}$ has order of contact $2$ at $-1$:
$$
\Re\left(\frac{1}{|\phi_{lp}'(-1)|} + \frac{-1\phi_{lp}''(-1)}{\phi_{lp}'(-1)|\phi_{lp}'(-1)|} - 1 \right) = \frac{1}{9} + \frac{80}{81} - 1  > 0.
$$
A computation shows that $\psi_1(z) = (4-3z)/(5-4z)$ is a selfmap of $\D$ whose second-order data at $1$ agrees with that of $\phi_{lp}$ at $1$ and $\psi_2(z) = (41z+32)/(40z + 49)$ is a selfmap of $\D$ that satisfies $D_2(\psi_2, -1) = D_2(\phi_{lp}, -1)$.  Thus by Theorem~\ref{KMS2T}, $C_{\phi_{lp}}$ and $C_{\psi_1} + C_{\psi_2}$ differ by a compact operator.
\end{example}
\subsection{Spectral preliminaries} Let $H$ be a complex Hilbert space and $T: H\rightarrow H$ be a bounded linear operator. 
The {\it spectrum} $\sigma(T)$ of $T$ is given by
$$
\sigma(T)= \{\lambda \in \C: T-\lambda I\ \text{is not invertible}\}.
$$
Observe that $1$ belongs to the spectrum of every composition operator---it is, in fact, an eigenvalue because if $f$ is a constant function, then $C_\phi f = f.$  

The {\it essential spectrum} $\sigma_e(T)$ of $T$ is given by
$$
\sigma_e(T)= \{\lambda \in \C: T-\lambda I\ \text{is not Fredholm}\}.
$$
Recall that $T: H\rightarrow H$ is {\it Fredholm} provided $T$ has closed range while ker$(T)$ and ker$(T^*)$ are finite dimensional.  Alternatively, the Fredholm operators are those representing invertible elements in the {\it Calkin Algebra}, $B(H)/B_0(H)$, where $B(H)$ is the collection of bounded linear operators on $H$ and $B_0(H)$ is the ideal of compact operators on $H$.   Thus $T\in B(H)$ is Fredholm if and only if $[T]:= T + B_0(H)$ is invertible in the Calkin Algebra.  

Throughout this paper, we let $[T]$ denote the equivalence class of $T$ in the Calkin algebra. Thus, e.g., we have $[C_{\phi_{lp}}] = [C_{\psi_1} + C_{\psi_2}] = [C_{\psi_1}] + [C_{\psi_2}]  $, where $\phi_{lp}$, $\psi_1$, and $\psi_2$ are given by (\ref{lpmap}) and (\ref{psi12}), respectively.

We let $r(T)=\max\{|\lambda|: \lambda \in \sigma(T)\}$ denote the spectral radius of $T$, and $r_e(T) = \max\{|\lambda|: \lambda\in \sigma_e(T)\}$, the essential spectral radius of $T$.  For Hardy-space composition operators, Cowen \cite[Theorem 2.1]{Cow2} proves the following:
\begin{quotation}
{\it Let $\omega$ be the Denjoy-Wolff point of $\phi$.  If $\omega\in \D$, then $r(C_\phi) = 1$; otherwise, $r(C_\phi) = \frac{1}{\sqrt{\phi'(\omega)}}$.}
\end{quotation}
If $\omega\in\partial \D$, then $r_e(C_\phi) = r(C_\phi)$ (see \cite[Lemma 5.2]{BSM}). If $\omega \in \D$, then useful formulas for $r_e(C_\phi)$ exist in case $\phi$ is analytic on the closed disk $\D^-$ (\cite{KM} and \cite[p.\ 296]{CMB}) or univalent on the open disk $\D$ (\cite{CMSP}).  For instance, if $\phi$ is univalent on $\D$, we have
$$
r_e(C_\phi) = \lim_{k\rightarrow \infty} \left(\sup\left\{\frac{1}{|(\phi^{[k]})'(\zeta)|}: \zeta\in \partial \D\right\}\right)^{1/(2k)}, 
$$
where $(\phi^{[k]})'(\zeta)$ is the angular derivative of $\phi^{[k]}$ at $\zeta$ (taken to be $\infty$ when it does not exist). Suppose that $\phi$ is linear-fractional; then the preceding essential-spectral radius formula is easy to apply.  If $\phi$ doesn't contact the boundary  (so that $\|\phi\|_\infty < 1$) or if it contacts the boundary at a point that is not fixed (so that $\|\phi^{[2]}\|_\infty < 1$), then, respectively, $C_\phi$ or $C_\phi^2 = C_{\phi^{[2]}}$ is compact, and, in either case, $r_e(C_\phi) = 0$. If $\phi$ fixes $\zeta_0\in \partial \D$, then $\phi'(\zeta_0)$ is positive, $(\phi^{[k]})'(\zeta_0) = \phi'(\zeta_0)^k$ for every $k\ge 1$, and $(\phi^{[k]})'(\zeta) = \infty$ for $\zeta\in \partial \D\setminus\{\zeta_0\}$; thus, in this case, $r_e(C_\phi) = 1/\sqrt{\phi'(\zeta_0)}$.  

As we just noted, if $\phi$ is linear fractional and takes $\zeta\in \partial\D$ to $\eta\in \partial \D$ and $\eta\ne \zeta$, then $\|\phi^{[2]}\|_\infty < 1$ and $C_\phi^2$ is compact.  We make frequent use of a generalization of this observation. Suppose that $\phi_1$ and $\phi_2$ are linear-fractional, non-automorphic selfmaps of $\D$ such that $\phi_1(\zeta_1) = \eta_1$, and $\phi_2(\zeta_2) = \eta_2$, where $\zeta_1, \zeta_2, \eta_1$, and $\eta_2$ belong to $\partial \D$;  then if $\eta_2 \ne \zeta_1$,   we have $\|\phi_1\circ\phi_2\|_\infty < 1$  so that $C_{\phi_2}C_{\phi_1}$ is compact.

We now summarize results of a number of authors  characterizing the spectra and essential spectra of composition operators on $H^2(\D)$ whose symbols are non-automorphic linear-fractional selfmaps of $\D$.
\begin{theorem}\label{LFCS} Let $\phi$ be a linear-fractional selfmap of $\D$ that is not an automorphism of $\D$ and let $\omega$ be its Denjoy-Wolff point.  
\begin{itemize}
\item[(a)]   Suppose $\omega \in \D$ and $\phi$ does not fix a point on $\partial \D$. Then either $C_\phi$ or  $(C_\phi)^2$ is compact and we have
$$
\sigma_e(C_\phi) = \{0\} \quad \text{and}\quad  \sigma(C_\phi) = \{\phi'(\omega)^n: n =0,1,2, \ldots\} \cup\{0\}.
$$
{\rm (Caughran and Schwartz \cite[Theorem 3]{CSw})}
\item[(b)] Suppose $\omega \in \D$ and $\phi$ fixes a point $\zeta_0$ on $\partial \D$. Then
$
\sigma_e(C_\phi) = \left\{z: |z| \le1/\sqrt{\phi'(\zeta_0)}\right\}$ and
$$
  \sigma(C_\phi) =\left\{z: |z| \le \frac{1}{\sqrt{\phi'(\zeta_0)}}\right\} \cup\{\phi'(\omega)^n: n = 0, 1, 2, \ldots\}.
$$
{\rm (Kamowitz \cite[Theorem A(3)]{KM}, Cowen \cite[Proof of Theorem 5 ]{Cow}; see also \cite[Theorem 3.2]{BSp}.)}
\item[(c)] Suppose that $\omega\in \partial \D$ and $\phi'(\omega) < 1$ {\rm (}so that $\phi$ is of hyperbolic type{\rm)}.  Then
$$
\sigma_e(C_\phi) = \sigma(C_\phi) = \left\{z: |z| \le \frac{1}{\sqrt{\phi'(\omega)}}\right\}.
$$
{\rm (See Deddens \cite[Theorem 3(iv)]{DD} and note that $C_\phi$ is similar to $C_{az + (1-a)}$ where $a = \phi'(\omega)$.  Each point in the punctured disk $\{z: 0< |z| < 1/\sqrt{a}\}$ is an infinite multiplicity eigenvalue of $C_{az + (1-a)}$ with eigenfunctions of the form $z\mapsto (1-z)^\beta$, where $\Re(\beta) > -1/2$. Thus the essential spectrum and spectrum coincide.)}
\item[(d)]  Suppose that $\omega\in \partial \D$ and $\phi'(\omega) = 1$ $($so that $\phi$ is of parabolic non-automorphism type$)$. Let $a= \omega\phi''(\omega)$, which has positive real part.  Then
$$
\sigma_e(C_\phi) = \sigma(C_\phi) = \left\{e^{-at}: t\ge 0\right\} \cup \{0\}.
$$
{\rm (Cowen (\cite[Corollary 6.2]{Cow2}) applied to $C_\psi$, where $\psi(z) = \bar{\omega}\phi(\omega z)$; $C_\psi$ is similar to $C_\phi$.)}
\end{itemize}
\end{theorem}

Note that in (d) of the preceding theorem we have equality of spectrum and essential spectrum because every point of the spectrum is a non-isolated boundary point of spectrum. See, e.g., \cite[Theorem 6.8, p.\ 366]{Con}. Note also we have the following easy corollary of Theorem~\ref{LFCS}.
\begin{cor}\label{DWDDES} Suppose that $\phi$ is a non-automorphic linear-fractional selfmap of $\D$ such that $\phi$ fixes a point $\zeta$ on $\partial \D$ and $\phi'(\zeta) >1$. Then  
$$
\sigma_e(C_\phi) = \left\{z:|z|\le \frac{1}{\sqrt{\phi'(\zeta)}}\right\}.
$$
In particular, the derivative of $\phi$ at its fixed point on $\partial \D$ completely determines the essential spectrum of $C_\phi$.  
\end{cor}
\begin{proof} Because $\phi$ is a non-automorphic linear-fractional selfmap of $\D$, it contacts $\partial\D$ only at $\zeta$ and because $\phi'(\zeta) > 1$, $\zeta$ can't be the Denjoy-Wolff point of $\phi$.  Thus, $\phi$ must have its Denjoy-Wolff point inside $\D$ and part (b) of Theorem~\ref{LFCS} now yields this corollary.
\end{proof} 

To illustrate Theorem~\ref{LFCS}, we consider the linear-fractional selfmaps $\psi_1$ and $\psi_2$ given by (\ref{psi12}), both of which are not automorphisms.  Because $\psi_1(1) =1$ and $\psi_1'(1) =1$, we see (d) applies.  Also $\psi_1''(1) = 8$. Thus $\sigma_e(C_{\psi_1}) = \sigma(C_{\psi_1}) = [0,1]$.   Note  $\psi_2(-1) = -1$ and $\psi_2'(-1) = 9$; so by Corollary~\ref{DWDDES}, we have $\sigma_e(C_{\psi_2}) = \left\{z: |z| \le \frac{1}{3}\right\}$.   Note $\psi_2(4/5)= 4/5$ so that the Denjoy-Wolff point of $\psi_2$ is $4/5$; also, $\psi_2'(4/5) = 1/9$. Thus by part (b) of Theorem~\ref{LFCS}, we have 
$$
\sigma_e(C_{\psi_2}) = \left\{z: |z| \le \frac{1}{3}\right\} \quad {\rm and}\quad \sigma(C_{\psi_2}) =  \left\{z: |z| \le \frac{1}{3}\right\}\cup \{1\}.
$$

We have seen that $[C_{\phi_{lp}}] = [C_{\psi_1}] + [C_{\psi_2}]$, that is, $C_{\phi_{lp}}$ and $C_{\psi_1} + C_{\psi_2}$ represent the same equivalence class in the Calkin algebra.  We are interested in the essential spectrum of $C_{\phi_{lp}}$; equivalently, the spectrum of the sum  $[C_{\psi_1}] + [C_{\psi_2}]$ in the Calkin Algebra.  Because $\psi_1$ and $\psi_2$ are linear-fractional non-automorphic maps fixing distinct points on the unit circle, both  $\psi_1\circ \psi_2$ and $\psi_2\circ \psi_1$ have $H^\infty(\D)$ norm less than $1$; hence $C_{\psi_1\circ\psi_2}$ and $C_{\psi_2\circ \psi_1}$ are compact.  Equivalently, if $T_1 = [C_{\psi_1}]$ and $T_2 = [C_{\psi_2}]$, then  $T_2T_1=0$ and $T_1T_2 = 0$.   Thus we see that the Calkin-algebra elements $T_1$ and $T_2$ satisfy the following annihilation relations:
\begin{equation}\label{BAR}
 T_1 T_2 = T_2 T_1 = 0.
 \end{equation}
 This observation motivates the results in the next section, which will be used in Section 4 to  characterize the spectra of composition operators with symbols in $\ess(2)$.   Proposition~\ref{TA}, e.g., shows that if $T_1$ and $T_2$ are elements of any unital algebra over a field satisfying (\ref{BAR}), then $\sigma(T_1 +T_2)\setminus\{0\} = \left(\sigma(T_1) \cup \sigma(T_2)\right)\setminus\{0\}$; hence, applying this result in the context of the Calkin Algebra, we see that the set of nonzero points in $\sigma_e(C_{\phi_{lp}}) = \sigma_e\left(C_{\psi_1} + C_{\psi_2}\right)$ equals the set of  nonzero points in $\sigma_e(C_{\psi_1}) \cup \sigma_e(C_{\psi_2})= [0,1] \cup \{z: |z| \le 1/3\} $. Because the essential spectrum is closed,  we conclude that the essential spectrum of $C_{\phi_{lp}}$ is $[0,1] \cup \{z: |z| \le 1/3\}$ (moreover, the essential spectrum of $C_{\phi_{lp}}$ equals its spectrum---see Section 4).

\section{Spectra of Sums Whose Summands Satisfy Certain Annihilation Conditions}

Throughout this section, $\Al$ denotes a unital algebra over a field.  

 \begin{lemma}\label{FL}  Suppose that  $a_1$ and $a_2$ are elements of $\Al$ such that $a_1a_2 = 0$.  Then
\begin{equation}\label{STI}
\sigma(a_1 + a_2) \subseteq  \sigma(a_1) \cup \sigma(a_2).
\end{equation}
 \end{lemma}
 \begin{proof} Suppose that $\lambda \notin \sigma(a_1) \cup \sigma(a_2)$, so that
both $a_1-\lambda I$ and $a_2-\lambda I$ are invertible.   Note that $0$ must belong to $\sigma(a_1) \cup \sigma(a_2)$ because $a_1a_2 = 0$; thus, $\lambda$ is nonzero.   Because $a_1a_2 = 0$, we have
$$
(a_1-\lambda I)(a_2 - \lambda I) = -\lambda(a_1+a_2 -\lambda I).
$$ 
  Thus, $a_1 + a_2 -\lambda I$ is invertible, being a product of invertible elements. 
We conclude that if $\lambda$ belongs to the spectrum of $a_1+a_2$, then $\lambda\in \sigma(a_1) \cup \sigma(a_2)$; that is, inclusion (\ref{STI}) holds.  
\end{proof}

The preceding lemma yields, by induction, the corollary below (cf. \cite[Lemma 3.5]{SY}).
\begin{cor}\label{FLC} Let $n\ge 2$ be an integer and let $a_1, a_2, \ldots, a_n$ be elements of $\Al$  such that whenever $i,j$ belong to  $\{1, 2, \ldots, n\}$ and $i < j$, 
$$
a_ia_j = 0.
$$
Then 
$$
\sigma\left(\sum_{k=1}^n a_k\right)  \subseteq \left( \bigcup_{k=1}^n \sigma(a_k)\right).
$$
\end{cor}

In an earlier version of this paper, all results in this section were stated for  sums $T_1 + T_2 + \cdots + T_n$ of Hilbert-space operators, where certain pairwise products from the set $\{T_1, T_2^*, T_2, T_2^*, \ldots, T_n, T_n^*\}$ are zero.   For example, the following proposition read as follows:  If $T$ and $S$ are operators on the Hilbert space $H$ such that $ST = TS = 0$ and $S^*T = TS^* = 0$, then $\sigma(S+T)\setminus \{0\} = \left(\vstrut \sigma(S) \cup \sigma(T)\right)\setminus\{0\}$.   Trieu Le (private communication) pointed out that the hypothesis $S^*T = TS^* = 0$ is unnecessary and provided the proof of the resulting improved proposition. This inspired the author to improve similarly all the results in the section.    

\begin{prop}\ {\rm [Trieu Le, private communication]} \label{TA}  Suppose that $a_1$ and $a_2$ are elements of $\Al$ such that $a_1a_2 = 0$ and $a_2a_1 = 0$.     Then
\begin{equation}\label{Eq1}
\sigma(a_1 + a_2)\setminus\{0\} = \left(\vstrut \sigma(a_1) \cup \sigma(a_2)\right)\setminus \{0\}.
\end{equation}
\end{prop}
\begin{proof}    Because $a_1a_2 = 0$, just as Lemma~\ref{FL}, we have
\begin{equation}\label{ProductRep}
(a_1-\lambda I)(a_2 - \lambda I) = -\lambda(a_1 + a_2 -\lambda I).
\end{equation} 
Suppose that $\lambda \ne 0$ and $\lambda\notin \sigma(a_1 + a_2)$. Then (\ref{ProductRep}) shows that $a_1-\lambda I$ is right invertible while $(a_2-\lambda I)$ is left invertible.  However, since $a_2a_1 = 0$, (\ref{ProductRep}) holds with $a_1$ and $a_2$ interchanged.  Thus each factor  $a_1-\lambda I$ and $a_2-\lambda I$ is invertible.  Hence, we have 
$$
  \left(\vstrut \sigma(a_1) \cup \sigma(a_2)\right)\setminus \{0\} \subseteq \sigma(a_1 + a_2)\setminus\{0\}.
$$
The reverse inclusion holds by Lemma~\ref{FL}.  
\end{proof}
    
\begin{cor}\label{CTA}  Let $n\ge 2$ be an integer and let $a_1, a_2, \ldots, a_n$ be elements of $\Al$ such that whenever $i,j$ are distinct elements in $\{1, 2, \ldots, n\}$,
$$
a_ia_j = a_ja_i = 0 
$$
Then 
$$
\sigma\left(\sum_{k=1}^n a_k\right) \setminus\{0\} = \left( \bigcup_{k=1}^n \sigma(a_k)\right) \setminus \{0\}.
$$
\end{cor}
\begin{proof} Let $\alpha_1 = \sum_{k=1}^{n-1} a_k$ and $\alpha_2 = a_n$ and observe that $\alpha_1$ and $\alpha_2$ satisfy the hypotheses of Proposition~\ref{TA}.  Thus, this corollary holds holds by induction.
\end{proof}

The preceding corollary improves Lemma 3.4 of \cite{SY}, whose statement is equivalent to the following:   Suppose that  $a_1, a_2, \ldots, a_n$ are normal elements of a $C^*$ algebra such that $a_ia_j = a_j a_i = 0$ whenever $i,j\in \{1, \ldots, n\}$ are distinct; then
$$
 \sigma\left(\sum_{k=1}^n a_k\right) \setminus\{0\} = \left(\cup_{k=1}^n \sigma(a_k)\right)\setminus\{0\}.
$$

\begin{lemma}\label{LIP} Suppose that $a_1$ and $a_2$ are elements of $\Al$ such that  $a_1a_2 = 0$ and $a_2^2 = 0$.   Then 
$$
\sigma(a_1 + a_2) \setminus  \{0\} = \sigma(a_1) \setminus \{0\}.
$$
\end{lemma}
\begin{proof}
Since $a_2^2 = 0$,  $\sigma(a_2) = \{0\}$.    Because $a_1a_2 =0$, we have
\begin{equation}\label{BAE}
(a_1 - \lambda I)(a_2 - \lambda I) = -\lambda(a_1 + a_2 - \lambda I).
\end{equation}
Assume $\lambda \ne 0$.  Since $\sigma(a_2) = \{0\}$, we see $(a_2-\lambda I)$  is  invertible, and it follows from (\ref{BAE}) that  $\lambda \in \sigma(a_1)$ if and only if $\lambda\in \sigma(a_1 + a_2)$,  which yields the lemma. \end{proof}

\begin{prop}\label{n2c} Suppose that $a_1$ and $a_2$ are elements of $\Al$ such that $a_1^2 = 0$ and $a_2^2 = 0$.    Then
\begin{equation}\label{INTEQ}
\sigma(a_1+ a_2)\setminus \{0\} =  \{\lambda:  \lambda^2 \in \sigma(a_1a_2)\}\setminus\{0\} =  \{\lambda:  \lambda^2 \in \sigma(a_2a_1)\}\setminus\{0\}.  
\end{equation}
\end{prop}

\begin{proof}  A computation shows that
\begin{equation}\label{SMT}
(a_1 +\lambda I)(a_1 + a_2 - \lambda I)(a_2 + \lambda I) = \lambda(a_1a_2 - \lambda^2 I).
\end{equation}
Assume that $\lambda \ne 0$ and note that because $a_1^2 = 0 = a_2^2$,  both factors $(a_1 + \lambda I)$ and $(a_2 +\lambda I)$ are invertible. Thus, equation (\ref{SMT}) tells us that $(a_1 + a_2 - \lambda I)$ is invertible if and only if    $(a_1a_2 - \lambda^2 I)$ is invertible.  This observation establishes the first equality of (\ref{INTEQ}). To obtain the second,  modify (\ref{SMT}) by  interchanging $a_1$ and $a_2$ or apply Jacobson's Lemma, which states that for any $a_1, a_2\in \Al$, the nonzero elements of $\sigma(a_1a_2)$ and $\sigma(a_2a_1)$ must coincide.
\end{proof}

 Proposition~\ref{n2c} provides the base case of the inductive proof of the next proposition.
  
\begin{prop}\label{RSM}  Let $n\ge 2$ and let $a_0, a_1, \ldots, a_{n-1}$ be elements of $\Al$ satisfying 
\begin{equation}\label{ARRSM}
a_ja_k = 0 \   \text {for all}\  j, k  \in \{0, \ldots, n-1\} \  \text{ except possibly when} \ k = (j+1) \mod n 
\end{equation}
Then 
 \begin{equation}\label{GENINT}
\sigma\left(\sum_{j=0}^{n-1} a_j \right)\setminus \{0\} =  \{\lambda:  \lambda^n \in  \sigma\left( \Pi_{j=0}^{n-1}a_j \right)\}\setminus\{0\}.
\end{equation}
\end{prop}
Note that by Jacobson's Lemma, the set on the right of (\ref{GENINT}) may be replaced by  $\left\{\lambda:  \lambda^n \in  \sigma\left(\Pi_{j=k}^{k+n-1} a_{j \mod n}\right)\right\}\setminus \{0\}$ for any $k\in \{1, 2, \ldots, n-1\}$.

\begin{proof}[Proof of Proposition~\ref{RSM}] 
 Proposition~\ref{n2c} gives the result when $n = 2$.  Our argument is inductive:  fix $i\ge 3$ and assume the proposition is valid when $n = i-1$.    Assume $\lambda \ne 0$.

 Suppose that the annihilation relations (\ref{ARRSM}) hold for $n = i$.   We then have
\begin{align*} 
(a_0 + \lambda I)\left(\sum_{j=0}^{i-1} a_j - \lambda I\right) & = a_0a_1 + \lambda \sum_{j=1}^{i-1} a_j  - \lambda^2 I\\
& = \sum_{j=0}^{i-2} \alpha_j - \lambda^2I,
\end{align*}
where $\alpha_0 = a_0a_1 + \lambda a_1$ and for $j\in\{1, \ldots, i-2\}$, $\alpha_j = \lambda a_{j+1}$.    Because $(a_0 + \lambda I)$ is invertible ($a_0^2 = 0$), we see that 
\begin{equation}\label{1S}
\lambda \in \sigma\left(\sum_{j=0}^{i-1} a_j \right) \ \text{iff}\ \lambda^2 \in \sigma\left(\sum_{j=0}^{i-2} \alpha_j\right).
\end{equation}
 Now observe that
 $$
\alpha_j\alpha_k = 0 \   \text {for all}\  j, k  \in \{0, \ldots, i-2\} \  \text{ except possibly when} \ k = (j+1) \mod (i-1).
$$
Thus, by our induction hypothesis,  
\begin{equation}\label{2S}
\lambda^2 \in \sigma\left( \sum_{j=0}^{i-2} \alpha_j\right) \text{iff} \  \lambda^{2(i-1)} \in   \sigma\left(\alpha_0\cdots \alpha_{i-2}\right).
\end{equation}
However, by definition of $\alpha_j$, we have $ \alpha_0\cdots \alpha_{i-2} = \lambda^{i-2}a_0a_1\cdots a_{i-1} + \lambda^{i-1}a_1\cdots a_{i-1}$.  Because  $\lambda^{i-1}a_1\cdots a_{i-1}$ squares to $0$ and $ \left(\lambda^{i-2}a_0a_1\cdots a_{i-1}\right) \left(\lambda^{i-1}a_1\cdots a_{i-1}\right) = 0$, Lemma~\ref{LIP} and (\ref{2S}) combine to show that 
\begin{equation}\label{3S}
\lambda^{2(i-1)} \in   \sigma\left(\alpha_0\cdots \alpha_{i-2}\right) \ \text{iff} \ \lambda^{2(i-1)} \in \sigma\left(\lambda^{i-2}a_0a_1\cdots a_{i-1}\right).
\end{equation}
Finally, it's easy to see that
\begin{equation}\label{4S}
\lambda^{2(i-1)} \in \sigma\left(\lambda^{i-2}a_0a_1\cdots a_{i-1}\right) \ \text{iff} \ \lambda^i \in \sigma\left(a_0 a_1 \ldots a_{i-1}\right).
\end{equation}
Combining the conclusions of (\ref{1S}) through (\ref{4S}), we see
$$
\sigma\left(\sum_{j=0}^{i-1} a_j\right)\setminus \{0\} =  \{\lambda: \lambda^i \in \sigma\left(a_0 a_1 \ldots a_{i-1}\right)\}\setminus\{0\},
$$
which completes our inductive proof.
\end{proof}

\section{Main Results}

We apply the results of the preceding section with $\Al$ equaling the Calkin Algebra $B(H^2(\D))/B_0(H^2(\D))$ to characterize the essential spectrum of composition operators on $H^2(\D)$ whose symbols belong to $\ess(2)$.  The complete spectrum is easily derived once the essential spectrum is known.  

Let $\phi$ belong to $\ess(2)$, so that $E(\phi) = F(\phi)$ is finite. Suppose that  $E(\phi)$ is empty, then $C_\phi$ is compact \cite{CMa}, and thus by \cite[Theorem 3]{CSw} the spectrum of $C_\phi$ consists of $0$ together with the terms of the sequence $\left(\phi'(\omega)^k\right)_{k=0}^\infty $, where $\omega$ is the Denjoy-Wolff  point of $\phi$ (which is necessarily contained in $\D$).  Next assume that $E(\phi)$ is not empty:  
$$
E(\phi) = \{\zeta_1, \ldots, \zeta_n\}
$$
for some positive integer $n$.     By definition of $\ess(2)$,  $\phi$ is $C^2(\zeta_j)$ for $j= 1, \ldots, n$, and there are linear-fractional selfmaps of $\D$, $\psi_1, \ldots, \psi_n$, such that 
\begin{equation}\label{Decomp}
[C_\phi] = [C_{\psi_1}] + \cdots + [C_{\psi_n}],
\end{equation}
where $\psi_j$'s second-order data agrees with $\phi$'s at $\zeta_j$ for $j = 1, \ldots, n$.

\subsection{Partitioning the points of $E(\phi)$ based on their orbits under $\phi$.}\label{FSS1}   Let $j\in \{1, \ldots, n\}$, where, as above, $n$ is the number of elements in $E(\phi) = F(\phi)$.  Consider the iterate sequence $\zeta_j, \phi(\zeta_j), \phi^{[2]}(\zeta_j), \ldots$.  Either there is 
\begin{itemize}
\item[(i)] an  integer $k$,  $1\le k \le n$ such that $\phi^{[k]}(\zeta_j)\notin E(\phi)$, or
\item[(ii)]  $\phi^{[k]}(\zeta_j) \in E(\phi)$ for all $k\in \{0,  1,  2\ldots, n\}$.  
\end{itemize}
Suppose that (ii) holds; then  because $E(\phi)$ has $n$ elements, we see that there is a least positive integer $m \le n$  such that $\phi^{[m]}(\zeta_j) = \phi^{[i]}(\zeta_j)$, for some $i$ satisfying  $0 \le i \le m-1$.  If $ i = 0$, then the iterate sequence $(\phi^{[k]}(\zeta_j)_{k=0}^\infty$ is periodic with fundamental period $m$, consisting of repetitions of the cycle $\{\zeta_j, \phi(\zeta_j), \ldots, \phi^{[m-1]}(\zeta_j)\}$.  If $i > 0$, then $(\phi^{[k]}(\zeta_j))_{k=0}^\infty$ is eventually periodic with fundamental period $m-i$, and from $\phi^{[i]}(\zeta_j)$ onward consists of repetitions of the cycle $\{\phi^{[i]}(\zeta_j), \phi^{[i+1]}(\zeta_j), \ldots, \phi^{[m-1]}(\zeta_j)\}$.   

\begin{obser} \label{CYOB} Suppose that $P$ is a cycle of $\phi$ of length $\ell$ that lies in $E(\phi)$.  If $\xi \in P$, then $\phi^{[\ell]}(\xi) = \xi$. Moreover,  for any two points $\xi_1$ and $\xi_2$ of $P$, $(\phi^{[\ell]})'(\xi_1) = (\phi^{[\ell]})'(\xi_2)$.  Thus, if $\ell > 1$, then $(\phi^{[\ell]})'(\xi)>1$ at every point $\xi$ of $P$;  otherwise, $\phi^{[\ell]}$ would have more than one Denjoy-Wolff point.
\end{obser}

Suppose that all points in $E(\phi)$ satisfy (i) and consider 
\begin{equation}\label{Itout}
[C_\phi^{n+1}] = [\left(C_{\psi_1}+ \cdots + C_{\psi_n}\right)^{n+1}],
\end{equation}
where $\psi_1, \psi_2, \ldots, \psi_n$ are the linear-fractional non-automorphic maps of (\ref{Decomp}).  
Expanding $\left(C_{\psi_1}+ \cdots + C_{\psi_n}\right)^{n+1}$, we obtain a sum of products of the form
\begin{equation}\label{PF}
C_{\psi_{j_1}}C_{\psi_{j_2}} \cdots C_{\psi_{j_{n+1}}}, 
\end{equation}
where for $i = 1, 2, \ldots, {n+1}$, ${j_i}\in\{1, \ldots, n\}$.  Note that the product (\ref{PF}) is a linear-fractional composition operator with non-automorphic symbol
\begin{equation}\label{PF2}
\nu:= \psi_{j_{n+1}}\circ \psi_{j_n}\circ \cdots \circ \psi_{j_1}.
\end{equation}
Recalling that the second-order data of $\psi_{j_k}$ agrees with that of $\phi$ at $\zeta_{j_k}$, we see the only point that  $\nu$  can possibly take to $\partial \D$ is $\zeta_{j_1}$.  Since $\psi_{j_1}(\zeta_{j_1}) = \phi(\zeta_{j_1})$, we have  $\nu(\zeta_{j_1}) = \left(\psi_{j_{n+1}}\circ \psi_{j_n}\circ \cdots\circ \psi_{j_2}\right) (\phi(\zeta_{j_1}))$.  There are two possibilities. Either  $\phi(\zeta_{j_1}) = \zeta_{j_2}\in E(\phi)$, in which case $\psi_{j_2}(\phi(\zeta_{j_1})) = \phi^{[2]}(\zeta_{j_1})$; or $\phi(\zeta_{j_1}) \ne \zeta_{j_2}$, in which case $\psi_{j_2}(\phi(\zeta_{j_1}))$ lies in $\D$  and $\nu$ must have $H^\infty(\D)$ norm strictly less than $1$.   Suppose that the former case holds: we have 
$\nu(\zeta_{j_1}) = \left(\psi_{j_{n+1}}\circ \psi_{j_n}\circ \cdots \circ \psi_{j_3}\right)(\phi^{[2]}(\zeta_{j_1}))$.     
Now we repeat the preceding argument based on whether or not $\phi^{[2]}(\zeta_{j_1})= \zeta_{j_3}$. If so, $\nu(\zeta_{j_1}) = \left(\phi_{j_{n+1}}\circ \psi_{j_n}\circ \cdots\circ \psi_{j_4}\right) (\phi^{[3]}(\zeta_{j_1}))$; if not,  $\|\nu\|_\infty < 1$.   Continue this process.  Because we are assuming that all points in $E(\phi)$ fall into category (i) above, there must be a least positive integer $k$, $k\le n$, such that  $\phi^{[k]}(\zeta_{j_1}) \notin E(\phi)$, and since $\psi_{j_{k+1}}(\phi^{[k]}(\zeta_{j_1}))\in \D$, we conclude that $\|\nu\|_\infty < 1$. Hence $C_\nu$ is compact.   Applying this analysis to every summand of the expansion of $\left(C_{\psi_1}+ \cdots + C_{\psi_n}\right)^{n+1}$, we see, via (\ref{Itout}) that 
$C_\phi^{n+1}$ is compact when every point of $E(\phi)$ satisfies the ``iterate-out'' condition (i) above. Being power compact, the essential spectrum of $C_\phi$ is the origin and its spectrum consists of the origin together with the eigenvalue sequence $\left(\phi'(\omega)^k\right)_{k=0}^\infty$, where $\omega\in \D$ is the Denjoy-Wolff point of $\phi$.   Thus we are interested in the situation where some points of $E(\phi)$ satisfy condition (ii) above, so that $E(\phi)$ contains at least one cycle under $\phi$.

Assuming that $E(\phi)$ does contain some cycles, we partition the points of $E(\phi) = \{\zeta_1, \ldots, \zeta_n\}$, according to (i) and (ii) above:
\begin{itemize}
\item[(a)]  ``iterate-out points":  those $\zeta\in E(\phi)$ for which (i) holds; 
\item[(b)]   periodic and  eventually periodic points: those $\zeta\in E(\phi)$ for which  (ii) holds.
\end{itemize}
We further partition the periodic and eventually periodic points of $E(\phi)$.  Periodic cycles are either disjoint or they coincide.  Let  $P_1, \ldots, P_{n_c}$ be the disjoint  cycles of $\phi$ lying in $E(\phi)$.   For each $j\in \{1, 2, \ldots, n_c\}$, associate with $P_j$ a possibly empty set $L(P_j)$ of ``lead-in points'' consisting of those points $\zeta$ satisfying (ii), which are not in $P_j$, but for which $\phi^{[k]}(\zeta) \in P_j$ for some $k\ge 1$.   
Letting $A$ be the (possibly empty) set of iterate-out points of $E(\phi)$, we may express $E(\phi)$ as the following disjoint union
\begin{equation}\label{POE}
E(\phi) = A\  \cup  {\bigcup_{j\in \{1, \ldots, n_c\}}} (L(P_j)\cup P_j).
\end{equation}
 
We illustrate the preceding partition with a concrete example.  

\begin{example} \label{PPEX} Consider 
\begin{equation}\label{Ex}
\phi_1(z) = \kappa(z)\gamma(z^2),
\end{equation}
where
$$
\kappa(z) = \frac{-z^3}{2-z^8}
$$
and
 $\gamma$ is the inner function
$$
\gamma(z) =  \frac{(1+i)+(3-i)z}{3+i +(1-i)z}.
$$
The right-halfplane incarnation  of $\gamma$, created via conjugation by $\rhp(z)= (1+z)/(1-z)$, is $\Gamma(w) = 2w + i$, and it follows that, e.g.,  $\gamma$ fixes $1$ and $-i$ and maps $-1$ to $i$.  It's clear that the only points in $\partial \D$ that $\kappa$ maps to the unit circle are the $8$ eighth-roots of unity.  Because $\gamma$ is an inner function,   these roots of unity are necessarily taken to the unit circle under $\phi_1$ (and are the only points of $\partial \D$ that are mapped to $\partial \D$ under $\phi_1$).   Because $\phi_1$ is analytic on the closed disk, it follows that $\phi_1$ has finite angular derivative at each of the eighth roots of unity.  Thus  $E(\phi_1)=  \{\zeta\in \partial\D: |\phi_1(\zeta)| = 1\} =\{\zeta_j: j = 1, \ldots, 8\}$, where $\zeta_j = e^{(j-1)\pi i/4}$ for $j = 1, 2, \ldots, 8$.
\begin{center}
\includegraphics[height=2.3in]{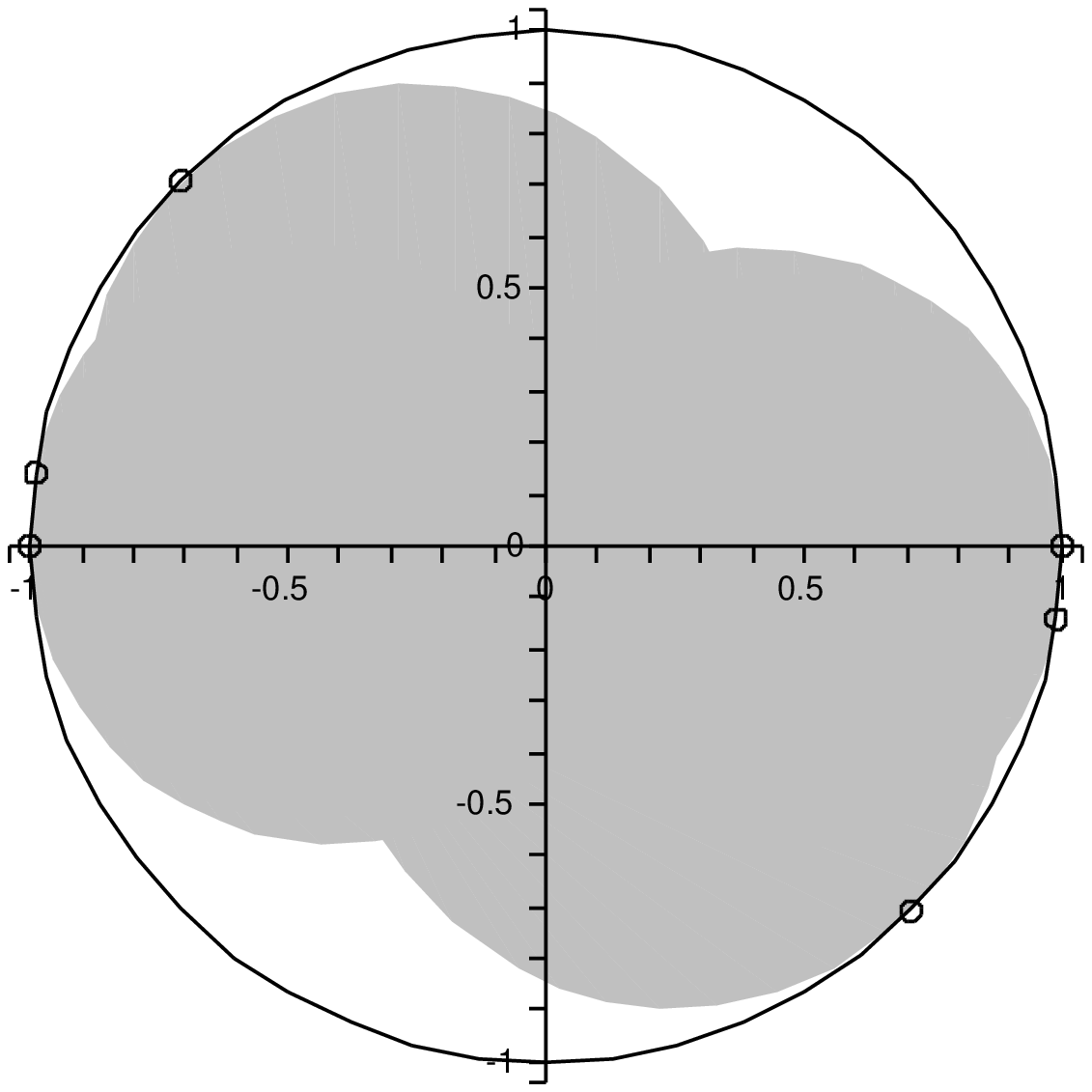}

\begin{minipage}{4in}  {\sc Figure}: $\phi_1(\D)$ is shaded gray, with images of points in $E(\phi_1)$ under $\phi_1$ circled.
\end{minipage}
\end{center}
We now classify the behavior of each of the points in $E(\phi_1)$ under iteration by $\phi_1$. 
\begin{itemize}
\item $\zeta_1 = 1$:   We have $\phi_1(1) = -1$ and $\phi_1(-1) = 1$.  Thus $1$ and $-1$ are periodic points of $\phi_1$ of period 2, with corresponding cycle $\{1, -1\}$.
\item $\zeta_2 = e^{i\pi/4}$:   $\phi_1(\zeta_2) = \sqrt{2}\left(\frac{7}{10} - \frac{1}{10} i\right)$, which is not an eighth root of unity. Thus, the boundary orbit of $\phi_1$ with initial point  $\zeta_2$ iterates out of $E(\phi_1)$.  
\item $\zeta_3= i$ : Now $\phi_1(i) = -1$, so that $\phi_1^{[2]}(i) = 1$ and $\phi_1^{[3]}(i) = -1$.  Thus $i$ is eventually periodic and leads into the cycle $\{1, -1\}$.  
\item $\zeta_4 = e^{i3\pi/4}$:  Note that  $\phi_1(\zeta_4) = \zeta_4$ so that $\zeta_4$ is fixed by $\phi_1$---it is a periodic point with period 1.  
\item $\zeta_5: = -1$: We've already observed that $-1$ is periodic of period $2$.
\item $\zeta_6 = e^{i5\pi/4}$:  It's easy to check that the boundary orbit of $\phi_1$ with initial point $\zeta_6$ iterates out of $E(\phi_1)$.  
\item $\zeta_7 = -i$:  We have $\phi_1(-i) = 1$, so that $-i$ is eventually periodic, leading into the cycle $\{1, -1\}$.  
\item $\zeta_8 = e^{i7\pi/4}$:  The point $\zeta_8$ is fixed by $\phi_1$.
\end{itemize}
Hence, in the notation of (\ref{POE}) above, we have for $\phi_1$ given by (\ref{Ex}),
\begin{align*}
A= \{e^{i\pi/4}, e^{i5\pi/4}\}, P_1= \{1, -1\}, L(P_1) = \{i, -i\},\\ P_2= \{e^{3\pi i/4}\}, L(P_2) = \{ \},  P_3=\{e^{7\pi i/4}\}, L(P_3) = \{ \}.
\end{align*}
\end{example}

We now return to the general situation where $E(\phi)$ has $n$ points $\{\zeta_1, \ldots, \zeta_n\}$ partitioned according to (\ref{POE}) and consider the corresponding partition of the summands of the decomposition of $[C_\phi]$ given by (\ref{Decomp}):  
\begin{equation}\label{ID}
[C_\phi] = \left[\vstrut \sum_{\{j: \zeta_j \in A\}} C_{\psi_j}\right] + \left[\sum_{k=1}^{n_c}  \left(\sum_{\{j: \zeta_j \in L(P_k)\cup P_k\}} C_{\psi_j}\right)\right].
\end{equation}
If $A$ is not empty, define 
\begin{equation}\label{TND}
T_0 =   \sum_{\{j: \zeta_j \in A\}} C_{\psi_j};
\end{equation}
 and for $k \in \{1, \ldots, n_c\}$, let 
\begin{equation}\label{TKD}
T_k = \sum_{\{j: \zeta_j \in L(P_k)\cup P_k\}} C_{\psi_j}.
\end{equation}   
 Thus (\ref{ID}) may be written (assuming $A\ne \emptyset$)
\begin{equation}\label{IDT}
[C_\phi] = [T_0] + \sum_{k=1}^{n_c} [T_k].
\end{equation}
Our goal is to characterize the essential spectrum of $C_\phi$ using the preceding decomposition.  Since $\phi$ is not an automorphism of $\D$, {\it $0$ necessarily belongs to the essential spectrum of $C_\phi$} (\cite[Theorem 1]{CTW}).  Thus, as we consider the potential contributions to the essential spectrum of the summands from (\ref{IDT}), a contribution of $\{0\}$ is meaningless.

\subsection{ Iterate-out points contribute nothing to the essential spectrum.}   Assume that $A\ne \emptyset$ and that $T_0$ and $T_k$ are defined as in (\ref{TND}) and (\ref{TKD}).
Let $k\in \{1, \ldots, n_c\}$ be arbitrary.  Let $j_1\in \{j: \zeta_j\in A\}$  and $j_2\in \{j: \zeta_j \in L(P_k) \cup P_k\}$  be arbitrary.   Observe that $C_{\psi_{j_1}} C_{\psi_{j_2}}$ is compact because $\psi_{j_2}\circ \psi_{j_1}$ has $H^\infty(\D)$ norm less than $1$. (The only point that the linear-fractional map $\psi_{j_1}$ takes to $\partial \D$ is $\zeta_{j_1}$ and $\psi_{j_1}(\zeta_{j_1}) = \phi(\zeta_{j_1}) \ne \zeta_{j_2}$; otherwise,  $\zeta_{j_1}$ would be eventually periodic under $\phi$, with all its iterates belonging to $E(\phi)$,  contradicting its membership in $A$.) It's also easy to see that $C_{\psi_{j_2}} C_{\psi_{j_1}}$ is compact. 

Let $S = T_0$, $T = \sum_{k=1}^{n_c} T_k$, and note we have just shown that $a_1:=[S]$ and $a_2:=[T]$ satisfy the annihilation hypotheses of Proposition~\ref{TA}.    Thus, we have
\begin{equation}\label{FFS}
\sigma_e(S + T) \setminus \{0\} =\left(\vstrut  \sigma_e(S) \cup \sigma_e(T)\right)\setminus \{0\}.
\end{equation}

We claim $\sigma_e(S) = \{0\}$.    Consider the operator $S^{n+1}: H^2(\D) \rightarrow H^2(\D)$.  It consists of a finite sum of products of $n+1$ composition operators, with each product having the form
\begin{equation}\label{NNPF}
C_{\psi_{i_1}}C_{\psi_{i_2}} \cdots C_{\psi_{i_{n+1}}}, 
\end{equation}
where for $k = 1, 2, \ldots, {n+1}$, $\zeta_{i_k}\in A$ and the integers $i_k, i_m$ (belonging to $\{j: \zeta_j\in A\}$) are not necessarily distinct.  Note that the product (\ref{NNPF}) is a composition operator with symbol
\begin{equation}\label{NNPF2}
\nu_0:= \psi_{i_{n+1}}\circ \psi_{i_n}\circ \cdots \circ \psi_{i_1}.
\end{equation}
   The only point that the linear-fractional selfmap $\nu_0$ of $\D$ can possibly take to $\partial \D$ is $\zeta_{i_1}$.  However, because $\zeta_{i_1}$ is an iterate-out point, the argument used above to show $\nu$ defined by (\ref{PF2}) satisfies $\|\nu\|_\infty<\infty$ shows $\|\nu_0\|_\infty < 1$.      Applying this analysis to every summand of $S^{n+1}$, we see that $S^{n+1}$ is compact and 
$$
\sigma_e(S) = \{0\}
$$
as claimed.

 Using (\ref{IDT}) and  (\ref{FFS}), we see if $A$ is not empty, then
\begin{equation}\label{OTK}
\sigma_e(C_\phi) \setminus \{0\} = \sigma_e(S + T)  \setminus \{0\} = \left(\vstrut \sigma_e\left(\sum_{k=1}^{n_c} T_k\right)\right)\setminus \{0\}.
\end{equation}
Note that  $\sigma_e(C_\phi)= \sigma_e\left(\sum_{k=1}^{n_c} T_k\right)$ obviously holds if $A$ is empty. 

\subsection{ Contributions to the essential spectrum from $T_1, \ldots, T_k$ are independent.} Now let  $m$ and $q$ be distinct indices in $\{1, \ldots, n_c\}$; then $L(P_m)\cup P_m$ and $L(P_q)\cup P_q$ are disjoint.  If  $j\in \{1, \ldots, n\}$ is such that $\zeta_j \in L(P_m)\cup P_m$ then $\psi_j(\zeta_j)\in L(P_m)\cup P_m$ and the same is true with $q$ replacing $m$.  This means that if $j_1,j_2\in \{1, 2,\ldots, n\}$ are such that $\zeta_{j_1}\in  L(P_m)\cup P_m$ and $\zeta_{j_2}\in  L(P_q)\cup P_q$, then $\psi_{j_1}\circ \psi_{j_2}$ and $\psi_{j_2}\circ \psi_{j_1}$ have $H^\infty(\D)$ norm less than $1$.  It follows that $[T_m][T_q] = 0$ and $[T_q][T_m] = 0$.    Thus, $[T_1], [T_2], \ldots, [T_{n_c}]$ satisfy the annihilation hypotheses of Corollary~\ref{CTA}, and hence
\begin{equation}\label{OTK2}
\sigma_e\left(\sum_{k=1}^{n_c} T_k\right)\setminus\{0\} = \bigcup_{k=1}^{n_c} \sigma_e(T_k)\setminus\{0\}.
\end{equation}
We now turn our attention to understanding $\sigma_e(T_k)$, focusing first on the ``cycle-based portion'' of the sum $T_k$, where we continue to assume that $T_k$ is defined by (\ref{TKD}).    

\subsection{Characterization of the essential spectrum of a cycle-based sum $\sum_{\{j:\zeta_j\in P_m\}} C_{\psi_j}$} \label{FSSC} Let $m\in \{1, \ldots, n_c\}$.  Independent of whether  the lead-in set $L(P_m)$ of $P_m$ is empty, we  characterize the essential spectrum of $\sum_{\{j:\zeta_j\in P_m\}} C_{\psi_j}$,  which is a sum of composition operators whose symbols correspond to the cycle $P_m$ of $\phi$.
Let $\ell$ be the length of the cycle $P_m$.  {\em Assume that $\ell > 1$.  }
 Let $\{j_0, \ldots, j_{\ell-1}\} = \{j: \zeta_j\in P_m\}$ and be such that $\phi(\zeta_{j_i}) = \zeta_{j_{(i+1)\mod\ell}}$ for $i = 0, 1, \ldots, \ell-1$.  Let $i,k\in \{0, \ldots, \ell-1\}$ be arbitrary (not necessarily distinct).  Observe that $\psi_{j_k} \circ \psi_{j_i}$ has $H^\infty(\D)$ norm less than $1$ unless $k =(i+1)\mod\ell$ (because the only point that $\psi_{j_i}$ maps to $\partial \D$ is $\zeta_{j_i}$ and $\psi_{j_i}(\zeta_{j_i}) = \zeta_{j_{(i+1)\mod\ell}}$).  Hence,
 \begin{quotation}
 $[C_{\psi_{j_i}}C_{\psi_{j_k}}] = 0$  for all   $i,k\in \{0, \dots, \ell-1\}$  except when  $k = (i+1)\mod\ell$.
 \end{quotation} 
Thus, the summands  of $\sum_{i=0}^{\ell-1} [C_{\psi_{j_i}}]$ satisfy the annihilation relations of Proposition~\ref{RSM}, and we may conclude that 
\begin{equation}\label{CPC}
\sigma_e(\sum_{\{j:\zeta_j\in P_m\}} C_{\psi_j}) \setminus\{0\} = \left\{\lambda: \lambda^\ell \in  \sigma_e\left(C_{\psi_{j_0}}C_{\psi_{j_1}} \cdots C_{\psi_{j_{\ell -1}}}\right)\right\}\setminus\{0\}.
\end{equation}

We now identify $ \sigma_e\left(C_{\psi_{j_0}}C_{\psi_{j_1}} \cdots C_{\psi_{j_{\ell -1}}}\right)$.  Observe that
$$
C_{\psi_{j_0}}C_{\psi_{j_1}} \cdots C_{\psi_{j_{\ell -1}}} = C_{\psi_{j_{\ell-1}}\circ \cdots \circ \psi_{j_1}\circ\psi_{j_0}}
$$
and that the linear-fractional, non-automorphic symbol $\gamma:=\psi_{j_{\ell-1}}\circ \cdots \circ \psi_{j_1}\circ\psi_{j_0}$ fixes the point $\zeta_{j_0}\in \partial \D$.  Moreover, 
$$
\gamma'(\zeta_{j_0}) = \Pi_{k=0}^{\ell - 1} \psi_{j_k}'(\zeta_{j_k}).
$$
Note that because of boundary data agreement, it's easy to see that
 \begin{equation}\label{GTP}
 \gamma'(\zeta_{j_0}) = (\phi^{[\ell]})'(\zeta_{j_0})
 \end{equation}
  and $\zeta_{j_0}$ is fixed for $\phi^{[\ell]}$.   Recall that we are assuming that $\ell > 1$. By Observation~\ref{CYOB},   $ \gamma'(\zeta_{j_0}) = (\phi^{[\ell]})'(\zeta_{j_0})$ exceeds $1$. Applying Corollary~\ref{DWDDES}, we have
  $$
\sigma_e(C_{\gamma}) = \left\{z: |z| \le  \frac{1}{\sqrt{\gamma'(\zeta_{j_0})}}\right\}.
$$
Thus, if the length $\ell$ of $P_m$ exceeds one, we have established
\begin{align*}
\sigma_e\left(\sum_{\{j:\zeta_j\in P_m\}} C_{\psi_j}\right)\setminus\{0\} & = \left\{\lambda: \lambda^\ell \in \sigma_e\left(C_{\gamma}\right)\right\}\setminus\{0\}  \quad ((\ref{CPC})\ \text{and the definition of}\ \gamma)\\
& = \left\{\lambda: \lambda^\ell \in  \left\{z: |z| \le  \frac{1}{\sqrt{\gamma'(\zeta_{j_0})}}\right\}\right\}\setminus\{0\} \\
&= \left\{\lambda: |\lambda|  \le  \left(\frac{1}{\gamma'(\zeta_{j_0})}\right)^{\frac{1}{2\ell}}\right\}\setminus\{0\} \\
& = \left\{\lambda: |\lambda|  \le  \left(\frac{1}{\left(\phi^{[\ell]}\right)'(\zeta_{j_0})}\right)^{\frac{1}{2\ell}}\right\}\setminus\{0\} , 
\end{align*}
where we have used (\ref{GTP}) to obtain the final equality. Because the essential spectrum is a closed set, 
$$
\sigma_e\left(\sum_{\{j:\zeta_j\in P_m\}} C_{\psi_j}\right) =  \left\{\lambda: |\lambda|  \le  \left(\frac{1}{\left(\phi^{[\ell]}\right)'(\zeta_{j_0})}\right)^{\frac{1}{2\ell}}\right\}.
$$
Note that by Observation~\ref{CYOB}, the preceding inequality holds with $\zeta_{j_0}$ being replaced by any point of $P_m$.

Finally, suppose that $\ell = 1$; then $P_m$ consists of a single element $\zeta_{j_0}$, which is fixed by $\phi$.  We have
$$
\sigma_e\left(\sum_{\{j:\zeta_j\in P_m\}} C_{\psi_j}\right) = \sigma_e(C_{\psi_{j_0}}),
$$
and the essential spectrum of the composition operator induced by the linear-fractional selfmap $\psi_{j_0}$ (whose second-order boundary data at $\zeta_{j_0}$ agrees with $\phi$'s) can be read off from Theorem~\ref{LFCS}. It will be either the disk
$\{z: |z| \le 1/\sqrt{\phi'(\zeta_{j_0})}\}$, in case $\phi'(\zeta_{j_0}) \ne 1$, or the spiral $\{e^{-at}: t\ge 0\} \cup\{0\}$, where $a = \zeta_{j_0}\phi''(\zeta_{j_0})$, if $\phi'(\zeta_{j_0}) = 1$.

\subsection{ Lead-in points contribute nothing to the essential spectrum.}\label{FSSF}  Suppose that $L(P_m)$ is not empty, containing $s$ elements, $\{\zeta_{L_1}, \ldots, \zeta_{L_s}\}$.  Call an element $\zeta_{L_i}$ of $L(P_m)$ {\it primitive} if there is no $\zeta_{L_q}\in L(P_m)$ such that $\phi(\zeta_{L_q}) = \zeta_{L_i}$.  Because $L(P_m)$ is finite and contains no periodic points for $\phi$, it must contain at least one primitive element. In fact, any element of $L(P_m)$ is either primitive or can be traced back to a primitive element of $L(P_m)$ through selection of inverse images under $\phi$. 

Let $\zeta_{L_i}$ be an arbitrary primitive element of $L(P_m)$.   Let $S = C_{\psi_{L_i}}$ and
$$
T = \sum_{\{j\ne L_i: \zeta_j\in (L(P_m) \cup P_m)\}} C_{\psi_j}.
$$
  Observe that $T$ is $T_m$ (defined by (\ref{TKD}))  with a single summand removed, namely that corresponding to $S$. Now observe that
$$
[S]^2 = 0 \ \text{and}\ [T][S] = 0,
$$
where  $[T][S] = 0$ by the primitivity of $\zeta_{L_i}$.  We apply Lemma~\ref{LIP} to conclude that 
$$
\sigma_e(T_m)\setminus\{0\} = \sigma_e(S + T) \setminus \{0\} =  \sigma_e(T)\setminus\{0\}
$$

Now if $L(P_m) \setminus \{\zeta_{L_i}\}$  not empty, then we repeat the argument of the preceding paragraph with $L(P_m)\setminus \{\zeta_{L_i}\}$ replacing $L(P_m)$ to obtain that $\sigma_e(T_m)\setminus\{0\} = \sigma_e(T)\setminus\{0\}$, where now $T$ is $T_m$ with two summands removed (corresponding to two lead-in points). We may continue this process to conclude that
\begin{equation}\label{OTK3}
\sigma_e(T_m) \setminus \{0\} =   \sigma_e\left(\sum_{\{j: \zeta_j\in P_m\}} C_{\psi_j}\right)\setminus\{0\}.
\end{equation}

\subsection{ Putting all the pieces together.}  

Combining the results from Subsections \ref{FSS1}  through \ref{FSSF}, we have the following:

\begin{theorem}\label{PPT}  Suppose that the analytic selfmap $\phi$ of $\D$ belongs to $\ess(2)$ while $E(\phi)= \{\zeta_1, \ldots, \zeta_n\}$ contains at least one  periodic cycle.  Let $P_1, P_2, \ldots, P_{n_c}$ be the $($disjoint$)$ periodic cycles contained in $E(\phi)$. Then
\begin{equation}\label{FMT}
\sigma_e(C_\phi) = \cup_{k=1}^{n_c} \sigma_e\left(\sum_{\{j: \zeta_j\in P_k\}} C_{\psi_j}\right),
\end{equation}
where for each $j\in \{1, \dots, n\}$,  $\psi_j$ is the linear-fractional selfmap of $\D$ such that $\psi_j$ and $\phi$ share the same second-order data at $\zeta_j$.
\end{theorem}
\begin{proof} Combine (\ref{OTK}), (\ref{OTK2}), and (\ref{OTK3}) to obtain the equality (\ref{FMT}) with $0$ excluded from both sides. However, $0$ belongs to the left side because $\phi$ is not an automorphism (\cite[Theorem 1]{CTW}) and $0$ belongs to the right side by the discussion of Subsection~\ref{FSSC} since the essential spectra on the right will either be disks containing $0$   or a  spiral containing $0$.  Thus (\ref{FMT}) holds.
\end{proof}

\begin{theorem}\label{MT}  Suppose that the analytic selfmap $\phi$ of $\D$ belongs to $\ess(2)$ while $E(\phi)= \{\zeta_1, \ldots, \zeta_n\}$ contains at least one  periodic cycle.  Let $P_1, P_2, \ldots, P_{n_c}$ be the $($disjoint$)$ periodic cycles contained in $E(\phi)$. For each $j\in \{1, \ldots, n_c\}$, let $\ell_j$ be the length of $P_j$ and let $\zeta_{s_j}$ denote some element of $P_{j}$. 
Define
\begin{equation}\label{DOP}
\rho =\max\left\{\left(\frac{1}{\left(\phi^{[\ell_j]}\right)'(\zeta_{s_j})}\right)^{\frac{1}{2\ell_j}}: j\in \{1, \ldots, n_c\}\right\},
\end{equation}
and note that the value of $\rho$ is independent of the choice of $\zeta_{s_j}$ in $P_j$  by Observation~\ref{CYOB}.
\begin{itemize}
\item[(a)] If $\phi$ has Denjoy-Wolff point $\omega$ in $\D$, then
$$
 \sigma_e(C_\phi) = \{z: |z| \le \rho\}
 $$
 and 
 $$
 \sigma(C_\phi) =  \{z: |z| \le \rho\}
  \cup \{\phi'(\omega)^m: m= 0, \ldots N-1\},
$$ where $N$ is the least positive integer for which $|\phi'(w)^N| \le \rho$.
\item[(b)] If $\phi$ has Denjoy-Wolff  point $\omega\in \partial \D$ and $\phi'(\omega)  < 1$, then  $\rho = 1/\sqrt{\phi'(\omega)}$ and
$$
\sigma(C_\phi) = \sigma_e(C_\phi) = \{z: |z| \le \rho\}.
$$
\item[(c)]  If $\phi$ has Denjoy-Wolff  point $\omega\in \partial \D$, $\phi'(\omega ) = 1$, and $j_*\in \{1, \ldots, n_c\}$ is such that  $P_{j_*} = \{\omega\}$ $($and thus  $\zeta_{s_{j_*}} =\omega)$,  then for 
$$
\rho_*=  \max\left\{\left(\frac{1}{\left(\phi^{[\ell_j]}\right)'(\zeta_{s_j})}\right)^{\frac{1}{2\ell_j}}: j\in \{1, \ldots, n_c\}\setminus\{j_*\} \right\}
$$
and  $a= \omega\phi''(\omega)$ $($which necessarily has positive real part$)$, we have
$$
\sigma(C_\phi) = \sigma_e(C_\phi) = \{z: |z| \le \rho_*\} \cup \{e^{-at}: t \ge 0\},
$$
where we take $\rho_* =0$ if $ \{1, \ldots, n_c\}\setminus\{j_*\}$ is empty $($equivalently, $n_c=1)$.  
\end{itemize}
\end{theorem}

  Remarks:  The characterization of the spectrum $\sigma$ in part (a),  in case $\phi$ is univalent or analytic on the closed disk, follows from  (\cite[Corollary 19]{CMSP} and \cite[Theorem A, Part 3]{KM}); however, the characterization of the essential spectrum as the full disk $\{z: |z| \le \rho\}$ is new even in these cases.  In part (b), in case $\phi$ is analytic on the closed disk, the characterization of the spectrum follows from \cite[Corollary 4.8]{Cow2} (see also the patch for p.\ 296, line -8 at \verb!http://www.math.iupui.edu/~ccowen/Errata.html!), and it's known \cite[Theorem 4.5]{Cow2} that in general the essential spectrum of $C_\phi$ for $\phi$  of hyperbolic type with Denjoy-Wolff point $\omega$  contains at least the annulus $\{z:  \sqrt{\phi'(\omega)} \le |z| \le 1/\sqrt{\phi'(\omega)}\}$.  Again,  the characterization of the essential spectrum as the full disk $\{z: |z| \le \rho\}$ is new.  Finally, the result of case (c) appears completely new (excluding, the situation where $\phi$ is itself linear fractional or, in certain cases, where it differs from a linear-fractional composition operator by a compact operator \cite{BSp}), and, as discussed in the introduction, settles a conjecture of Cowen's \cite[Conjecture 4]{Cow2}.  
  \begin{proof}[Proof of Theorem \ref{MT}]  We rely on equation (\ref{FMT}):
  $$
\sigma_e(C_\phi) = \cup_{k=1}^{n_c} \sigma_e\left(\sum_{\{j: \zeta_j\in P_k\}} C_{\psi_j}\right),
$$
where for each $j\in \{1, \dots, n\}$,  $\psi_j$ is a linear-fractional selfmap of $\D$ such that $\psi_j$ and $\phi$ share the same second-order data at $\zeta_j$.
  
    Case (a). Suppose that $\phi$ has its Denjoy-Wolff point in $\D$; then  the same is true of $\phi^{[\ell_k]}$ for $k\in \{1, \ldots, n_c\}$.  Thus, for $k\in \{1, \dots, n_c\}$, the derivative of  $\phi^{[\ell_k]}$ at its fixed point $\zeta_{s_k}$,  must exceed $1$, and by the discussion of Subsection~\ref{FSSC}, 
   $$
   \sigma_e\left(\sum_{\{j: \zeta_j\in P_k\}} C_{\psi_j}\right) = \left\{z: |z| \le \left(\frac{1}{\left(\phi^{[\ell_k]}\right)'(\zeta_{s_k})}\right)^{\frac{1}{2\ell_k}}\right\}.
   $$
 It follows that $\rho < 1$ (where $\rho$ is defined by (\ref{DOP})).   From (\ref{FMT}),  it follows that $\sigma_e(C_\phi) = \{z: |z|\le \rho\}$.  Thanks to the continuity of the Fredholm index, any spectral point outside the essential spectrum (which is necessarily in the unbounded component of the complement of the essential spectrum) must be an eigenvalue of $C_\phi$.  Work of K\"{o}nigs \cite{Kg} shows that the nonnegative integral powers of $\phi'(\omega)$ are the only possible eigenvalues of $C_\phi$; moreover, these powers are spectral points (see, e.g. \cite[Theorem 4.1]{Cow2}), which completes the proof of (a).   
    
  Case (b). Suppose that $\phi$ has DW point $\omega\in \partial \D$ and $\phi'(\omega)  < 1$. Since the Denjoy-Wolff point $\omega$ is a periodic point of period $1$, there is some $j_*\in \{1, \ldots, n_c\}$, such that $P_{j_*} = \{\omega\}$, which means $\ell_{j_*} = 1$, $\zeta_{s_{j_*}} = \omega$, and $1/\sqrt{\left(\phi^{[\ell_{j_*}]}\right)'(\zeta_{s_{j_*}})} = 1/\sqrt{\phi'(\omega)} >1$.  For all $j\in \{1, \ldots, n_c\}\setminus\{j_*\}$, we must have $\left(\phi^{[\ell_j]}\right)'(\zeta_{s_j}) > 1$; otherwise, $\phi^{[\ell_j]}$ would have different Denjoy-Wolff points $\zeta_{s_j}$ and $\omega$. Thus $\rho = 1/\sqrt{\phi'(\omega)}$ as claimed.  By (\ref{FMT}), we have  $\sigma_e(C_\phi) = \{z: |z| \le \rho\}$ and since the spectral radius of $C_\phi$ is $1/\sqrt{\phi'(\omega)}$, this completes the proof of part (b).
  
Case (c). Suppose that  $\phi$ has DW point $\omega\in \partial \D$ and  $\phi'(\omega ) = 1$.   Since the Denjoy-Wolff point $\omega$ is a periodic point of period $1$, there is some $j_*\in \{1, \ldots, n_c\}$, such that $P_{j_*} = \{\omega\}$, which means $\zeta_{s_{j_*}} = \omega$.  Note that $a:=\omega\phi''(\omega)$  has positive real part (by Remark (a) following Proposition~\ref{OCT} since $\phi$ has second-order contact at $\omega$). We have
$$
\sigma_e\left(\sum_{\{j: \zeta_j\in P_{j_*}\}} C_{\psi_j}\right)= \sigma_e(C_{\psi_{s_{j_*}}}) = \{0\} \cup \{e^{-at}: t\ge 0\},
$$
where the second equality follows from from part (d) of Theorem~\ref{LFCS} because $D_2(\psi_{s_{j_*}},\zeta_{s_{j_*}}) = D_2(\phi, \zeta_{s_{j_*}})$.  If $n_c =1$, we have verified that the essential spectrum of $C_\phi$ is correctly characterized by part (c).  Suppose that $n_c > 1$. For each
 $k\in \{1,\ldots, n_c\}\setminus\{j_*\}$, we have (from the discussion of Subsection~\ref{FSSC})
$$
\sigma_e\left(\sum_{\{j: \zeta_j\in P_{k}\}} C_{\psi_j}\right) =  \left\{z: |z|  \le  \left(\frac{1}{\left(\phi^{[\ell_k]}\right)'(\zeta_{s_k})}\right)^{\frac{1}{2\ell_k}}\right\}.
$$
Moreover, because, for every $k\in \{1,\ldots, n_c\}\setminus\{j_*\}$, the fixed point $\zeta_{s_k}$ of $\phi^{[\ell_k]}$ is not the Denjoy-Wolff point of $\phi^{[\ell_k]}$, we see $\rho_*=  \max\left\{\left(\frac{1}{\left(\phi^{[\ell_j]}\right)'(\zeta_{s_j})}\right)^{\frac{1}{2\ell_j}}: j\in \{1, \ldots, n_c\}\setminus\{j_*\}\right\}$ is less than $1$ and
(\ref{FMT}) may be applied to obtain the characterization of the essential spectrum described in part (c).  To see that the spectrum equals the essential spectrum, we note that the complement of the essential spectrum has one (unbounded) component.  The only spectral points in the unbounded component of the essential resolvent must be isolated eigenvalues.  However, Proposition~2.7 of \cite{BSp} shows that no eigenvalue of $C_\phi$  (other than possibly $1$, which belongs to the essential spectrum) can be an isolated point of the spectrum of $C_\phi$. Thus $\sigma_e(C_\phi) = \sigma(C_\phi)$ and the proof of part (c) is complete. \end{proof}

\subsection{Applications of Theorem~\ref{MT}}
We have already discussed how the preceding theorem applies to the composition operator $C_{\phi_{lp}}$ where
$\phi_{lp}(z) =  (2z^2 - z -2)/(2z^2-3)$.   By the discussion of Example~\ref{MLPS2}, $\phi_{lp}$ belongs to $\ess(2)$.   The cycles of $\phi_{lp}$ lying in $E(\phi)$ are $P_1 = \{1\}$ and $P_2 = \{-1\}$.  Since $\phi'(1) = 1$, $\phi''(1) = 8$ and $\phi'(-1) = 9$, by part (c) of Theorem~\ref{MT},
$$
 \sigma(C_{\phi_{lp}}) = \sigma_e(C_{\phi_{lp}}) = \left\{z: |z| \le \frac{1}{\sqrt{\phi'(-1)}}\right\} \cup \{e^{-8t}: t \ge 0\} =  \left\{z: |z| \le \frac{1}{3}\right\} \cup [0, 1].
 $$
Let's consider some additional applications.

Let $\phi(z) = \frac{-z}{3-2z^2}$.   Since $\{z: |\phi(z)| = 1\} = \{-1, 1\}$  and  $\phi$ has finite angular derivative  at $1$ and $-1$, $E(\phi) = \{1, -1\}$.  Since $\phi$ is analytic on the closed disk, we obviously have $\phi\in C^2(1) \cap C^2(-1)$.  Finally, it's easy to use Proposition~\ref{OCT} to see that $\phi$ has second order of contact at both $-1$ and $1$.  Thus $\phi\in \ess(2)$.  Here, $E(\phi) = \{-1,1\}$ is a single periodic cycle and $0$ is the Denjoy-Wolff point of $\phi$.  Thus by part (a) of Theorem~\ref{MT},  $\sigma_e(C_\phi) = \{z: |z| < 1/\sqrt[4]{(\phi^{[2]})'(1)}\} = \{z: |z| < 1/\sqrt{5}\}$ and since the Denjoy-Wolff derivative $\phi'(0) = -1/3$, we have $\sigma(C_\phi) = \{z: |z| < 1/\sqrt{5}\} \cup \{1\}$.

Our next example features an analytic selfmap of $\D$ that does not extend to be analytic on a neighborhood of the closed disk $\D^-$.  Let $\sqrt{\cdot}$ denote the principal branch of the square-root function.  Consider the selfmap $\phi$ of $\D$ whose right-halfplane incarnation $\Phi$ is given by
$$
\Phi(w) = 2w + 1 - \frac{1}{\sqrt{w+1}}.
$$
It's easy to check that the unit-disk incarnation of $\Phi$ is
$$
\phi(z) = \frac{2\sqrt{2}(1+z)- (1-z)\sqrt{1-z}}{4\sqrt{2} - (1-z)\sqrt{1-z}}.
$$
 Here $E(\phi) = \{-1, 1\}$ and $\phi\in C^2(-1)$ (in fact $\phi$ is analytic in a neighborhood of $-1$, and it's easy to check directly that $\phi\in C^2(1)$).  Using the second-order boundary data at $-1$ ($\phi(-1) = -1$, $\phi'(-1) = 5/2$, $\phi''(-1) = -33/8$) and at $1$ ($\phi(1) = 1$, $\phi'(1) = 1/2$ and $\phi''(1) = 0$), one can use Proposition~\ref{OCT} to see that $\phi$ has second-order contact at $-1$ and $1$.  Thus $\phi\in \ess(2)$.  The set $E(\phi)$ consists of two fixed points with $1$ being the Denjoy-Wolff point of $\phi$.  Applying Theorem~\ref{MT}(b), we have $\sigma_e(C_\phi) = \sigma(C_\phi) = \{z: |z| \le \sqrt{2}\}$.

Our final application of Theorem~\ref{MT} is to the selfmap $\phi_1$ of Example~\ref{PPEX}.  As we discussed earlier,  $E(\phi_1)=  \{\zeta\in \partial\D: |\phi_1(\zeta)| = 1\} =\{\zeta_j: j = 1, \ldots, 8\}$ where $\zeta_j = e^{(j-1)\pi i/4}$ for $j = 1, 2, \ldots, 8$.  Proposition~\ref{OCT} may be used to confirm that $\phi$ has second-order contact at each point of $E(\phi)$.   Because $\phi$ has Denjoy-Wolff point $0\in \D$, part (a) of Theorem~\ref{MT} holds.  The function $\phi_1$ has 3 periodic cycles $P_1= \{1, -1\},  P_2= \{e^{3\pi i/4}\},  P_3=\{e^{7\pi i/4}\}$.  We compute,  $(\phi^{[2]})'(-1) = (\phi^{[2]})'(1) = 144 $, $\phi'(e^{3\pi i/4}) = 15 $ and $\phi'(e^{7\pi i/4}) = 15$. Thus $\rho =1/\sqrt{12} $ and we conclude that $\sigma_e(C_{\phi_1}) = \{z: |z| \le 1/\sqrt{12}\}$.  Since the Denjoy-Wolff derivative of $\phi$ is $\phi'(0) = 0$, we have $\sigma(C_{\phi_1}) = \{z: |z| \le 1/\sqrt{12}\} \cup \{1\}$.

\subsection{Two Open Questions}Observe that the results of Theorem~\ref{MT} are consistent with ``yes'' answers to the following open questions concerning  spectra of composition operators on $H^2(\D)$. 
  \begin{itemize}
   \item  For $\phi$ of hyperbolic type or  of parabolic type, do the spectrum and essential spectrum of $C_\phi$ always coincide?
 \item   Let $\phi$ be an non-automorphic analytic selfmap of $\D$ having its Denjoy-Wolff point $\omega$ in $\D$.   Does the essential spectrum consist of a disk (possibly degenerate) of radius less than $1$? 
 \end{itemize}


\end{document}